\documentclass[a4paper]{article}
\usepackage{amsmath}
\usepackage{amssymb}
\usepackage{graphicx}
\usepackage{amsthm}
\usepackage{dsfont}
\usepackage{bm}
\usepackage{siunitx}
\usepackage{colonequals}

\usepackage[    backend=bibtex,   style=alphabetic,maxalphanames=4,minalphanames=4 ,maxbibnames=5]{biblatex}
\usepackage{graphicx}
\usepackage[dvipsnames]{xcolor}
\usepackage[color=black!40,textcolor=black!60,bordercolor=black!40,backgroundcolor=white,textsize=small]{todonotes}

\usepackage{tikz}
\usetikzlibrary{patterns}
\setuptodonotes{noline}
\usepackage[utf8]{inputenc}
\usepackage[toc,page]{appendix}
\usepackage[margin=1in]{geometry}

\input{preamble.tex}

\numberwithin{equation}{section}

\title{Computational lower bounds for multi-frequency group synchronization}

\usepackage{authblk}

\author[1]{Anastasia Kireeva\thanks{Email: \texttt{anastasia.kireeva@math.ethz.ch}.}}

\author[1]{Afonso S.\ Bandeira\thanks{Email: \texttt{bandeira@math.ethz.ch}.}}

\author[2]{Dmitriy Kunisky\thanks{Email: \texttt{dmitriy.kunisky@yale.edu}. Partially supported by ONR Award N00014-20-1-2335 and a Simons Investigator Award to Daniel Spielman.}}

\affil[1]{Department of Mathematics, ETH Z\"{u}rich}
\affil[2]{Department of Computer Science, Yale University}

\date{June 5, 2024}

\begin{document}

\maketitle
\begin{abstract}
    We consider a group synchronization problem with multiple frequencies which involves observing pairwise relative measurements of group elements on multiple frequency channels, corrupted by Gaussian noise. We study the computational phase transition in the problem of detecting whether a structured signal is present in such observations by analyzing low-degree polynomial algorithms. We show that,  assuming the low-degree conjecture, in synchronization models over arbitrary finite groups as well as over the circle group $SO(2)$, a simple spectral algorithm is optimal among algorithms of runtime $\exp(\tilde{\Omega}(n^{1/3}))$ for detection from an observation including a constant number of frequencies. 
    Combined with an upper bound for the statistical threshold shown in Perry et al.\ \cite{perryOptimalitySuboptimalityPCA2016}, our results indicate the presence of a statistical-to-computational gap in such models with a sufficiently large number of frequencies. 
\end{abstract}

\clearpage

\tableofcontents

\clearpage

\section{Introduction}

Identifying and recovering a hidden structured object from noisy matrix-valued observations is a classical problem in statistics and machine learning \cite{lee2007nonlinear,hastie2009elements}.
Furthermore, many such problems incorporate a significant amount of group structure which relates to the underlying physics or symmetry of the input data in the problem.
Such applications include problems in electron microscopy, image processing, computer vision, and others. 
In the specific task of group synchronization, the goal is to recover unknown group elements from their noisy pairwise measurements. 
Besides its practical importance, this task also involves a mathematically intriguing combination of algebraic structure and statistical inference. 

Orientation estimation in cryo-electron microscopy (cryo-EM) serves as an instructive example \cite{Singer_CryoEM} of a synchronization problem. Cryo-EM is a technique for the analysis of three-dimensional properties of biological macro-molecules based on their two-dimensional projections. In order to reconstruct the molecule density, one needs (after some processing of the initially two-dimensional data) to determine unknown rotations $g_u \in SO(3)$ from noisy measurements of their relative alignments $g_u g_v^{-1}$. Other examples include community detection in graphs (which can be cast as synchronization over $\bb Z_2$) \cite{Abbe_z2,Deshpande,Javanmard_z2}, multireference alignment in signal
processing (which involves synchronization over $\bb Z_L$) \cite{bandeira2014multireference}, network clock synchronization \cite{giridhar2006distributed}, and many others \cite{cucuringu2012sensor,bandeira2014multireference,peters2015sensor}. %

In a general synchronization problem over a group $G$, one aims to recover the group-valued vector $u = (g_1, \dots, g_n) \in G^n$ from noisy pairwise information about $g_k g_j^{-1}$ for all (or, in some cases, a subset of) pairs $(k, j)$. A natural way to model this is to postulate that we obtain a function of $g_k g_j^{-1}$, corrupted with additive Gaussian noise,
\begin{equation}\label{eq:fgigjinv}
Y_{kj} = f(g_k g_j^{-1}) + W_{kj}
\end{equation}
for i.i.d.\ Gaussian random variables $W_{kj}$. Henceforth we will focus on the setting where measurements are available for all pairs $(k, j)$.

We will return to this general setting later, but for now we focus on the specific case of angular synchronization, where the objective is to determine phases $\varphi_1, \dots, \varphi_n \in [0, 2\pi]$ from their noisy relative observation $\varphi_k - \varphi_j \mod 2 \pi$ \cite{singer2011angular,Bandeira_angularsynch}. This problem can be seen as synchronization over $SO(2)$, or equivalently, over the complex circle group $U(1) = \{e^{i \varphi}, \varphi \in [0, 2 \pi)\} $.
We denote these isomorphic groups by
\begin{equation}
    \bb S \colonequals SO(2) \cong U(1) .
\end{equation}

Each pairwise alignment between elements $k$ and $j$ is expressed as $e^{i (\varphi_k - \varphi_j)}$, and the obtained noisy observation is
$$
Y_{kj} = \lambda  x_k \bar{x}_j + W_{kj}.
$$
Here $x_k = e^{i \varphi_k}$ and  $\bar{x}_j = e^{-i\varphi_j}$ denotes the complex conjugate. The scalar parameter $ \lambda$ is a \emph{signal-to-noise ratio}, and $W_{kj}$ is Gaussian white noise as above. In this case, the observation can also be seen as a rank-one perturbation of the Wigner random matrix $W$, 
\begin{equation}\label{eq:informal_synch_def}
Y = \lambda x x^* + W.
\end{equation}
This model is also referred to as a \emph{Wigner spiked matrix model}, has been studied extensively in the literature, and admits a sharp phase transition in the feasibility of estimating or detecting $x$ dictated by a variant of the Baik--Ben Arous--P\'{e}ch\'{e} (BBP) transition \cite{BBP,FeralPeche,BenaychGeorges_RaoNadakuditi}. %
Above a certain critical value $\lambda > \lambda_*$, detection is possible based on the top eigenvalue of the observation matrix.
Moreover, the top eigenvector of $Y$ correlates non-trivially with the true signal $x$. 
Below the threshold, when $\lambda < \lambda_*$, one cannot detect the signal reliably from the top eigenvalue and eigenvector as the dimension grows to infinity. %
The method of extracting the signal information from the top eigenvalue and its corresponding eigenvector is often referred to as \emph{principal component analysis (PCA)}, and the corresponding threshold $\lambda_*$ as the \emph{spectral threshold}. The PCA estimator does not take into account any structural prior information we might have on the signal, such as sparsity or entrywise positivity. Consequently, the spectral threshold depends only on the $\ell^2$ norm of the signal $\|x\|$,
and for some choices of a prior distribution of $x$, the performance of PCA is sub-optimal compared to the algorithms exploiting this structural information about the signal \cite{Zou_sparsePCA,dAspremont_sparsePCA,Johnstone_sparsePCA,Montanari2014NonNegativePC}. 

Nevertheless, while it is possible to improve on PCA for some choices of sparse priors, for many dense priors, no algorithm can beat the spectral threshold \cite{Deshpande,perryOptimalityPCA2016JSTOR}. Examples of such settings for synchronization problems cast as spiked matrix models include $\bb Z_2$ synchronization, angular synchronization, and other random matrix spiked models. 
This situation changes when considering a model with multiple frequencies, where the addition of frequencies introduces more signal information, thus potentially lowering the threshold at which detection or estimation become feasible.

In this work, we consider obtaining measurements through several frequency channels, which is motivated by the Fourier decomposition of the non-linear objective of the non-unique games problem \cite{Bandeira_NonUniqueGames_2020} (see \Cref{rmk:noisy indicators}). 
In the angular synchronization case, this translates to obtaining the following observations:
\begin{equation}\label{eq:intro_ang_def}
\left\{\begin{split}
    Y_1 &= \frac{\lambda}{n} x x^* + \frac{1}{\sqrt{n}}W_1,\\
    Y_2 &= \frac{\lambda}{n} x^{(2)} (x^{(2)})^* +\frac{1}{\sqrt{n}} W_2,\\
    &\vdots\\
    Y_L &= \frac{\lambda}{n} x^{(L)} (x^{(L)})^* +\frac{1}{\sqrt{n}} W_L.
\end{split}
\right.\end{equation}
where $x^{(k)}$ denotes the entrywise $k$th power, and $W_1, \dots, W_L$ are independent noise matrices whose off-diagonal entries have unit variance (refer to \Cref{def:angular_synch_U1} for a precise definition). With the scaling above, PCA (that is, computing the largest eigenvalue) succeeds in detecting a signal in any one of the $Y_i$ past the spectral threshold $\lambda > 1$. %

One may expect, that by combining information over the $L$ frequencies, it ought to be possible to detect the signal reliably once $\lambda > 1/\sqrt{L}$. 
This intuition comes from the fact that given $L$ independent draws of a single frequency, PCA would indeed detect the signal once $\lambda > 1/\sqrt{L}$. 
This suggests the question: do independent observations of $L$ \emph{different} frequencies provide as much information about the signal as $L$ independent observations of the \emph{same} frequency?
Because of the extra algebraic structure of the multi-frequency model, the phase transitions are not well-understood even in the case of two frequencies.
It is at least known that the above hope is too good to be true: while our intuition would lead us to believe that it should be possible to detect the signal from two frequencies once $\lambda > 1 / \sqrt{2} \approx \num{0.707}$, actually it is provably impossible (information-theoretically; that is, with unbounded computational budget) for any $\lambda < \num{0.937}$ \cite{perryOptimalitySuboptimalityPCA2016}.
On the other hand, once $\lambda > 1$, then the signal can be detected from any one of the $Y_i$ using PCA.
Thus multiple frequencies certainly are not as useful as independent observations of a single frequency.
This suggests another question: are multiple frequencies useful \emph{at all}?
That is, in this setting, even with $L = 2$ frequencies, is it possible to detect the signal for any $\lambda < 1$?

The same questions may be asked for synchronization over a finite group $G$.
The results described below hold for arbitrary finite groups in the setting to be described in \Cref{sec:finite_groups}, but for the sake of concreteness we may consider synchronization over the cyclic group $G = \bb Z / L \bb Z \eqqcolon \bb Z_L$ with all frequencies excluding the trivial one and taking one per conjugate pair (this will correspond to non-redundant irreducible representations). The number of such frequencies amounts to $\floor{L/2}$. This model may be viewed as a discrete variant of angular synchronization where $\varphi_i \in \{0, \frac{1}{L} 2\pi, \dots, \frac{L - 1}{L}2\pi\}$.
For such a synchronization model, the work of \cite{perryOptimalitySuboptimalityPCA2016} showed a similar impossibility result: it is impossible to detect the signal once %
\begin{equation}\label{eq:stat-lb}
    \lambda < \sqrt{\frac{2(L-1) \log(L-1)}{L (L-2)}}
\end{equation}
for $L > 2$ and $\lambda < 1$ for $L=2$. In particular, it implies that the statistical threshold $\lambda_{\text{stat}}(\bb Z_L)$ is bounded below by the right-hand side value of \eqref{eq:stat-lb}. 
Conversely, they also showed that, for sufficiently large $L$, there exists an \emph{inefficient} algorithm for detection that succeeds once 
\begin{equation}
    \label{eq:stat-ub}
    \lambda > \sqrt{\frac{4 \log L}{L - 1}} \gg \frac{1}{\sqrt{L}}.
\end{equation}
As $L$ grows, these two bounds provide a tight characterization of the scaling $\Theta(\sqrt{\log L / L})$ of the statistical threshold for this problem. 
Moreover, once $L \geq 11$, the quantity in \eqref{eq:stat-ub} is smaller than 1, and thus there is a computationally inefficient algorithm that is superior to PCA applied to a single frequency.
However, as for angular synchronization, it remained unknown whether this algorithm could be made efficient, and more generally what the limitations on computationally efficient algorithms are in this setting.
Consequently, this motivates the following question for group synchronization problems in general:
\begin{center} 
\emph{Does detection by an efficient algorithm become possible at a lower signal-to-noise ratio compared to a single-frequency model?}
\end{center} 

Using non-rigorous derivations from statistical physics together with numerical computations, the authors of \cite{Perry_AMP_synch} predicted that the answer is negative. This conjecture is predicated on the optimality of AMP-type algorithms, which do not always capture the optimal threshold \cite{wein2019kikuchi}. In this work, we derive computational lower bounds for the synchronization problems over $\bb S$ and over finite groups using the low-degree polynomials framework. %
Before presenting the main result for synchronization over $\bb S$, we first formalize the concept of reliable detection. This concept involves distinguishing a probability measure with a planted signal from a measure containing only pure noise. In our case, the latter corresponds to the distribution of the noise matrices of the corresponding structure, or, equivalently, of observations $Y_1, \dots, Y_L$ where $\lambda$ is set to zero. 
\begin{definition}\label{def:strong-detection}
    We say that a sequence of measurable functions $f_n: \mathcal S \to \{\rm p, \rm q\}$ achieves \emph{strong detection} between a sequence of pairs of probability measures $\bb P_n$ and $\bb Q_n$ if
\begin{align*}
&\text{if }Y \sim \bb P_n \text{ then } f_n(Y) = \rm p \quad \text{with probability } 1- o(1);\\
&\text{if }Y \sim \bb Q_n \text{ then } f_n(Y) = \rm q \quad \text{with probability } 1- o(1).
\end{align*}
\end{definition}
\noindent
In words, the test function $f_n$ is such that, in the limit of $n \to \infty$, the probability of the test function making a mistake (either a Type~I or Type~II error, in statistical language) is diminishing.

\begin{theorem}[$\bb S$-synchronization lower bound; informal]\label{thm:angular_comp_tr_intro}
    Consider the angular synchronization model $\eqref{eq:intro_ang_def}$ with $L$ frequencies, where $L$ is a constant that does not depend on $n$. 
    If the Low-Degree Conjecture holds (see \Cref{sec:low-degree-intro}), then for any $\lambda \le 1$, any algorithm for strong detection requires runtime at least $\exp(\tilde{\Omega}(n^{1/3}))$.
\end{theorem}

To formulate our results over general finite groups $G$, we consider the Peter-Weyl decomposition (a generalization to general compact groups of the Fourier decomposition) of an observation of the form $f(g_kg_j^{-1})$.
This leads to the viewpoint of observing the signal through noisy observations of its image under the irreducible representations of $G$. Informally, each pairwise measurement corresponding to a representation $\rho$ is a block matrix with blocks given by
\begin{equation}\label{eq:def_mfreq_groups_intro}
Y_{kj}^{\rho} = \frac{\lambda}{n} \rho(g_k) \rho(g_j^{-1}) + \frac{1}{\sqrt{n}} W_{kj} \text{ for each } \rho \in \Psi,
\end{equation}
where $W_{kj}$ is a Gaussian noise of appropriate type and covariance matrix. Each observation $Y_{kj}$ can be scalar or matrix depending on the dimension of representation $\rho$. If the list of representations $\Psi$ contains all irreducible representations excluding the trivial one, and, for complex representations, taking only one per conjugate pair, then our main result is as follows.

\begin{theorem}[Finite group synchronization lower bound; informal]
    Consider the Gaussian synchronization model over a finite group $G$ of size $L$ over all irreducible representations (excluding the trivial and redundant ones, as above), where $L$ is a constant that does not depend on $n$.
    If the Low-Degree Conjecture holds (see \Cref{sec:low-degree-intro}), then for any $\lambda \leq 1$, any algorithm for strong detection requires runtime at least $\exp(\tilde{\Omega}(n^{1 / 3}))$.
\end{theorem}

\begin{remark}[Equivalent model: noisy indicators]\label{rmk:noisy indicators}%

A natural model for receiving information about $g_kg_j^{-1}$ is to receive a ``score'' (such as the log-likelihood, or a common-lines score in Cryo-EM) $z_{kj}(g)$ for each possible group element $g\in G$ measuring how likely it is that $g_kg_j^{-1} = g$.
This is a general synchronization model, and is the motivation behind the non-unique games (NUG) approach to estimation, where, if $z_{kj}$ is the log-likelihood, one would solve an optimization problem such as
\[
\max_{g_1,\dots,g_n} \sum_{kj}z_{kj}(g_kg_j^{-1}),
\]
corresponding to the maximum likelihood of estimating $g_1, \dots, g_n$ when the noise in pairwise measurements is independent.

We observe that this approach is connected to the multi-frequency model described above in terms of representations. Indeed, assume that $G$ is of size $L$ and let $g^\ast = (g_1^\ast, \dots, g_n^\ast)$ be the ground truth-signal. Consider
the score function as a noisy indicator of a form 
$$
z_{kj}(g) = \begin{cases}
    \gamma + w_{kj}(g) & \text{ if } g = g_k^\ast (g_j^\ast)^{-1}, \\
    w_{kj}(g) & \text{otherwise},
\end{cases}
$$
where $\gamma > 0$ and $w_{kj}(g)\sim\mathcal{N}(0,1)$. 

In this case, when $\gamma = \sqrt{2} \lambda \sqrt{L / n}$, the noisy indicator model is equivalent to the multi-frequency model~\eqref{eq:def_mfreq_groups_intro}. The basic idea of this transformation is to change basis to the ``matrix coefficients'' of the irreducible representations of $G$ invoking the Peter-Weyl theorem.
We give the detailed proof of the equivalence in \Cref{sec:noisy_indicator}.
Note that this model treats all group elements symmetrically and does not involve a notion of ``closeness'' of group elements such as closeness along the complex unit circle for $SO(2)$.
After the above change of basis, this corresponds to including \emph{all} non-equivalent irreducible representations of $G$ among the observations rather than just a subset as in~\eqref{eq:def_mfreq_groups_intro}.

\end{remark}

While the precise statistical threshold value remains undetermined, existing upper bounds (the demonstration by \cite{perryOptimalitySuboptimalityPCA2016} of an inefficient algorithm for testing in some cases when $\lambda < 1$) combined with the result of the present study establishes the presence of a statistical-to-computational gap in the low-degree sense in models with a sufficiently large number of frequencies.

\subsection{Related work}
The synchronization model with multiple frequencies was first formally introduced in the work \cite{perryOptimalitySuboptimalityPCA2016}. %
The authors explored statistical distinguishability for this model. As it is natural to expect, the statistical threshold is strictly less than the spectral one when the number of frequencies is sufficiently large. %
As a reminder, the main results of that work are summarized in \Cref{fig:phase_tran_finiteL}. 

\begin{figure}[t]
    \centering
\begin{tikzpicture}
    \def\startX{0}
    \def\endX{14}
    \def\tickZero{0.5}
    \def\tickOne{4}
    \def\tickStat{6}
    \def\tickTwo{7.4}
    \def\tickThree{12}
    
    \fill[red!25] (\startX, -0.1) rectangle (\tickOne, 1);
    \fill[red!10] (\tickOne, -0.1) rectangle (\tickStat, 1);
   \fill[yellow!10] (\tickStat, -0.1) rectangle (\tickTwo, 1);
    \fill[yellow!25] (\tickTwo, -0.1) rectangle (\tickThree, 1);
    \fill[green!20] (\tickThree, -0.1) rectangle (\endX, 1);
    
    \draw[->, thick] (\startX,0) -- (\endX,0) node[right]  {$\lambda$};

    \draw (\tickZero,0.1) -- (\tickZero,-0.1); 
    \node[align=center] at (\tickZero, -0.5) {$\frac{1}{\sqrt{L}}$};
    \draw (\tickOne,0.1) -- (\tickOne,-0.1); 
    \node[align=center] at (\tickOne, -0.5) {$\sqrt{\frac{2(L-1) \log(L-1)}{L(L-2)}}$};%
    \node[align=center] at (\tickOne, -1.2) { \small \cite{perryOptimalitySuboptimalityPCA2016}};
    \draw (\tickTwo,0.1) -- (\tickTwo,-0.1);
    \node[align=center] at (\tickTwo, -0.5) {$\sqrt{\frac{4\log L}{L-1}}$ };
    \node[align=center] at (\tickTwo, -1.2) { \small \cite{perryOptimalitySuboptimalityPCA2016}};
    \draw (\tickStat,0.1) -- (\tickStat,-0.1);
    \node at (\tickStat, -0.5) {$\lambda_{\text{stat}} = \text{?}$};
     \draw (\tickThree,0.1) -- (\tickThree,-0.1);
     \node[align=center] at (\tickThree, -0.5) {$\lambda_{\text{comp}} = 1$};%
     
    \node[align=center] at (\tickThree, -1.2) {\small \cite{Perry_AMP_synch} (AMP) \\ \small (this paper, low-degree)};

\pgfmathsetmacro\labelOne{(\startX + \tickOne)/2}
    \pgfmathsetmacro\labelTwo{(\tickOne + \tickTwo)/2}
    \pgfmathsetmacro\labelThree{(\tickTwo + \tickThree)/2}
    \pgfmathsetmacro\labelFour{(\tickThree + \endX)/2}

    \node at (\labelOne, 0.5) {impossible};
    \node at (\labelThree, 0.5) {possible but hard};
    \node at (\labelFour, 0.5) {easy};
\end{tikzpicture}
    \caption{A schematic illustration of phase transitions in the multi-frequency synchronization model over a finite group of order $L$ (for sufficiently large $L$). $1 / \sqrt{L}$ on the far left is included as it is the ``optimistic'' though incorrect computational threshold we would expect if independent observations of different frequencies behaved like independent observations of the same frequency. The \emph{``possible but hard''} regime corresponds to the statistical-to-computational gap in the low-degree sense. In this signal-to-noise regime, strong detection is information-theoretically possible; however, conjecturally, there are no efficient algorithms achieving it.}
    \label{fig:phase_tran_finiteL}
\end{figure}
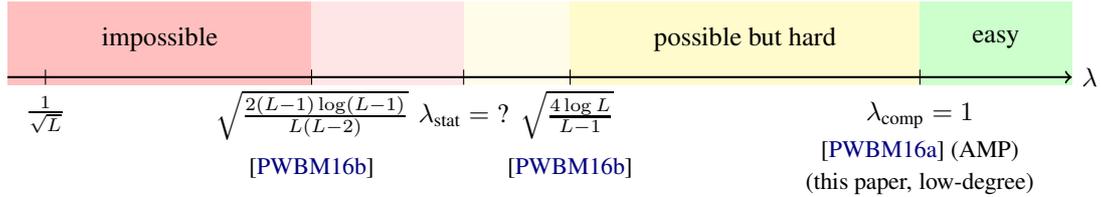

For the signal recovery, several approaches have been developed that improve accuracy compared to PCA by utilizing information from additional frequencies. One of the approaches is based on approximate message passing (AMP) algorithms adapted specifically for multiple frequencies, \cite{Perry_AMP_synch}. This method can be used for estimation and detection. The empirical results and non-rigorous derivations from statistical physics suggest that the computational threshold coincides with the spectral one. In this work, we confirm this prediction in the low-degree polynomials sense. 

The recent work \cite{Gao2019MultiFrequencyPS} proposes another efficient algorithm that leverages spectral information from each frequency channel. Compared to AMP it produces accurate estimates even under the spectral threshold. The discrepancy with our theoretic result is explained by the fact that we consider the setting of constant number of frequencies. Still, the empirical results of \cite{Gao2019MultiFrequencyPS} suggest that when number of frequencies diverges with the signal dimension, the computational threshold may supersede the spectral one. 

Finally, \cite{Wang2022MultiFrequencyJC} considers a multi-frequency synchronization model over a graph focusing on the joint estimation of the underlying community structure and estimation of the phases. The authors propose a spectral method based on the multi-frequency QR decomposition. In this work, we focus only on the setting when all pairwise observations are available, in other words, the observation graph is complete. 

\begin{remark}
As we were finalizing our manuscript, a related paper by Yang et al. became available, which analyzes phase transitions of the inference of the multi-frequency model \cite{yang2024asymptotic}. This study provides a rigorous computation of a replica formula for the asymptotic signal-observation mutual information that characterizes the information-theoretical limits of inference. 
Through the analysis of the replica formula, the authors also derive conjectured phase transitions for computationally-efficient algorithms.
    
    These results, like those of \cite{Perry_AMP_synch} are also predicated on the optimality of AMP-based methods. In our work, we establish the low-degree polynomial lower bounds for the multi-frequency model over finite groups and $SO(2)$. Our lower bounds for detection coincide with the lower bounds given in \cite{yang2024asymptotic}.
\end{remark}

\subsection*{Notation}
\addcontentsline{toc}{subsection}{Notation}

We use subscripts in the expectation, such as $\mathbb{E}_{x}$, to indicate the variables with respect to which the expectation is taken. We omit the subscript when it is clear from the context, or when we take the full expectation with respect to all present randomness. For a finite set $M$, $|M|$ denotes its cardinality. To describe the order of growth of functions, we use standard asymptotic notation like $o(\cdot), O(\cdot), \Theta(\cdot)$, $\Omega(\cdot)$, and so forth, which is always associated with the limit of the signal dimension $n\to\infty$.
The similar notations $\tilde{o}(\cdot), \tilde{O}(\cdot), \tilde{\Theta}(\cdot), \tilde{\Omega}(\cdot)$ refer to the same bounds up to polylogarithmic factors in $n$.
For a positive integer $K$, we write $[K] = \{1, \dots, K\}$.

\subsection*{Acknowledgements} 
\addcontentsline{toc}{subsection}{Acknowledgements}

Generative AI, spell-checking, search engines, and other similar tools were occasionally used by the authors to assist with the writing of this paper; the errors are all human.

\section{Average-case complexity from low-degree polynomials}

Understanding whether a statistical problem is computationally hard can be a very difficult task. %
Even for deterministic problems we cannot test against all possible algorithms, since there may exist algorithms we are not aware of. %
In classical computational complexity theory, the goal is to study the complexity of worst-case problems, i.e., instances of the problem with the most unfortunate input configuration, which prevents an algorithm from utilizing any favourable structure of the input. However, such configurations may occur very rarely in practice in statistical problems. 
Our focus therefore is in determining when \emph{typical} instances can be solved efficiently, which we refer to as \emph{average-case hardness}. By typical instances we refer to configurations occurring with high probability under the model's probabilistic assumptions.

Many classical statistical algorithms, including many settings of maximum-likelihood estimation, are NP-hard in the worst-case. %
Still, in the average case, there are often tractable procedures that can solve the problem exactly or with high precision with high probability (e.g., \cite{Abbe_z2, Abbe_recovery}). While these algorithms may produce a sub-optimal or completely wrong output on specific pathological instances, the probability of encountering those worst-case configurations is small; in other words, the worst case is not typical. %

However, analyzing computational hardness in the average case requires different tools than the worst case. 
There are several different existing approaches to tackling this task. First, we can adapt the classical complexity theory theory idea of reducing problems to other supposedly hard problems. This method has to be suitably adapted to the randomness involved in statistical problems, but has found success in many situations \cite{brennan2018reducibility, brennan2019optimal, brennan2020reducibility}. 
Another line of work instead studies the limitations of specific powerful classes of algorithms. A non-exhaustive list includes statistical query algorithms \cite{kearns1998efficient,feldman2017statistical}, approximate message passing (AMP) algorithms \cite{DonohoAMP}, local algorithms \cite{gamarnik2019landscape,ben2020algorithmic}, and the low-degree polynomial algorithms that we consider in this work. 

Originally, the low-degree polynomials framework arose in the study of Sum-of-Squares hierarchy of semidefinite programs \cite{Hopkins_lowdeg, hopkins2017power, Hopkins2017EfficientBE}. Subsequently, it has been developed into an independent method for studying computational hardness for detection and extended to different statistical tasks \cite{Rush_LowDegPlantedvsPlanted,Schramm_Lowdeg_recovery}. 
The idea of this ``low-degree method'' is to consider the special class of algorithms that can be expressed as polynomials of low degree. The framework provides a particular criterion that suggests if a problem, in our case a hypothesis testing problem, can be reliably solved using these low-degree polynomial algorithms. 

For a large array of classical problems, low-degree polynomials turn out to be as powerful as all known polynomial-time algorithms \cite{ding2023subexponential,Hopkins2017EfficientBE,Kunisky2019NotesLowDeg}, which makes them a powerful theoretical tool for studying computational tractability. Moreover, the failure of low-degree polynomials also implies failure of statistical query methods \cite{brennan2020reducibility}, spectral methods \cite{Kunisky2019NotesLowDeg}, and AMP methods under additional assumptions \cite{Montanari_AMPLowDeg}. By establishing connections to the free energy calculations common in statistical physics, it has also been shown that low-degree hardness implies the failure of local Monte-Carlo Markov chains algorithms \cite{bandeira2022franz}. %
It is conjectured that low-degree polynomials are as powerful as all polynomial-time algorithms for ``sufficiently nice'' distribution of problem instances for a wide range of problems. We refer to \cite{Hopkins_lowdeg} for a more formal statement of the conjecture and further details.

In this work, we will use this framework to analyze computational complexity of testing whether the signal is present in our observations under to the multi-frequency Gaussian synchronization model.
In the rest of this section, we describe the low-degree method more formally, starting first with the basic notions of statistical hypothesis testing. This short introduction follows the exposition given in the survey on the low-degree method of \cite{Kunisky2019NotesLowDeg}, to which we refer the reader for more details and additional examples. 

\subsection{Statistical distinguishability}
Let $\bb P_n$ and $\bb Q_n$ for $n \in \bb N$ be two sequences of probability distributions defined over a common sequence of measurable spaces $\mathcal S = ((\mathcal S_n, \mathcal F_n))_{n \in \bb N}$. Assume additionally that $\bb P_n$ is absolutely continuous with respect to $\bb Q_n$ for each $n$. We will refer to $\bb P_n$ as the \emph{planted distribution} and $\bb Q_n$ as the \emph{null distribution}. In our setting, the planted distribution often corresponds to the probability distribution of the signal perturbed by noise, while the null distribution describes the pure noise distribution. In statistical language, these are the null and alternative hypotheses, respectively. Finally, we think of the parameter $n$ measuring the size of the problem. 

Suppose we obtain a sample, which we refer to as an \emph{observation}, from either $\bb P_n$ or $\bb Q_n$. The goal is to determine which underlying distribution it was drawn from based on the information about the distributions and the observation itself.
Of course, unless the supports of the distributions are non-intersecting, it is impossible to differentiate the distributions for all possible samples, hence we are interested only in achieving the correct identification with high probability. This describes the notion of \emph{strong distinguishability}. 
\begin{definition}\label{def:contiguity}
    We say that $\bb P_n$ and $\bb Q_n$ are \emph{strongly distinguishable} if there exists any sequence of measurable functions $f_n: \mathcal S \to \{\rm p, \rm q\}$ achieving strong detection (in the sense of Definition~\ref{def:strong-detection}).
\end{definition}
\noindent
In words, we have strong distinguishability when we can find a test function $f_n$, such that in the limit of $n \to \infty$, the probability of the test function making a mistake decreases to zero.

This question is well-studied in asymptotic statistics, and, in the absence of limitations on computational resources, Le Cam's theory of contiguity provides powerful a tool for such analysis \cite{cam1960locally}.
A simple version of one of the main tools of this theory, sometimes called a \emph{second moment method}, is as follows.
\begin{lemma}\label{lm:lecam}
    Define the \emph{likelihood ratio} of $\bb P_n$ and $\bb Q_n$ as
    $$
    L_n(Y) \colonequals \frac{\rm d \bb P_n}{\rm d \bb Q_n }(Y).
    $$
    If $\|L_n\|^2 \colonequals \E_{Y \sim \bb Q_n} [L_n(Y)^2] $ remains bounded as $n \to \infty$, then there is no test strongly distinguishing $\bb P_n$ and~$\bb Q_n$. 
\end{lemma}

\subsection{Low-degree likelihood ratio (LDLR)}\label{sec:low-degree-intro} 

We next present the version of this tool that pertains to the possibility distinguishing the two distributions using only $f_n$ that are low-degree polynomials, which, per the assumptions of the low-degree framework, is a proxy for the possibility of distinguishing the distributions using \emph{any} polynomial-time algorithm.

This \emph{computational distinguishability} is governed by a low-degree polynomial analog of the likelihood ratio; in fact, the correct analog is simply the projection of the likelihood ratio to low-degree polynomials. 

\begin{definition} Let $\bb P_n$, $\bb Q_n$ be as before, and let $L_n$ denote likelihood ratio of $\bb P_n$ and $\bb Q_n$.
Define the \emph{degree-$D$ likelihood ratio} $L_n^{\le D}$ as the projection of likelihood ratio $L_n$ to the linear subspace of polynomials $\mathcal S_n \to \R$ of degree at most $D$, i.e., 
    $$
    L_n^{\le D} \colonequals \mathcal{P}^{\le D} L_n,
    $$
    where $\mathcal{P}^{\le D}$ is an orthogonal projection operator to the subspace of polynomials described above with respect to the inner product $\langle p, q \rangle = \bb E_{Y \sim \bb Q_n} p(Y) q(Y)$.
\end{definition}

\noindent
Polynomials of degree roughly $\log n$ are conjectured in the low-degree framework to be a proxy for polynomial-time algorithms. The following is a version of the main conjecture of this framework that encodes this hypothesis \cite{Hopkins_lowdeg,Hopkins2017EfficientBE,hopkins2017power}.

\begin{conjecture}\label{conj:low-degree}
    For ``sufficiently natural'' sequences of distributions $\bb P_n$, $\bb Q_n$, if $\|L_n^{\le D}\|^2 = \E_{Y \sim \bb Q_n}   [L_n(Y)^2] $ for $D = D(n) \ge (\log n)^{1 + \eps}$ remains bounded as $n \to \infty$, then there is no polynomial-time algorithm that achieves strong detection between $\bb P_n$ and $\bb Q_n$.
\end{conjecture}
\noindent
More generally, for larger $D$, degree-$D$ polynomials are believed to be as powerful as all algorithms of runtime $n^{\tilde{O}(D)}$, which corresponds to the time complexity of evaluating naively a degree-$D$ polynomial term by term. In particular, a stronger version of Conjecture~\ref{conj:low-degree} states that degree-$D$ polynomials continuously capture computational hardness for subexponential-time algorithms: if $\|L_n^{\leq D}\|^2$ is bounded as above, then strong detection requires runtime at least $\exp(\tilde{\Omega}(D))$.

\section{Synchronization over the circle group: main results}

We start the exposition with the angular synchronization model with multiple frequencies. 
The angular synchronization, also known as phase synchronization, concerns recovery of phases $\varphi_1, \dots, \varphi_n$ from potentially noisy relative observations $\varphi_k - \varphi_j \mod 2\pi$ \cite{singer2011angular,Bandeira_angularsynch}. The model can also be seen as synchronization over $SO(2)$. For this model we assume the uniform prior, i.e., each element $x_j$ of the signal $x \in \C^n$  follows $x_j \sim \Unif(U(1))$. Equivalently, we can define the prior as sampling the phase uniformly from $[0, 2\pi)$: $\varphi_j \sim \Unif([0, 2\pi))$ and setting $x_j = e^{i \varphi_j}$. The coordinates of the signal are sampled independently. This prior corresponds to the Haar measure in $SO(2)$.

Before defining the angular synchronization model, we first need to define a Gaussian ensemble to use as the noise model.
\begin{definition}\label{def:goe}
The \emph{Gaussian orthogonal} or \emph{unitary ensembles} (\emph{GOE} or \emph{GUE}, respectively) are the laws of the following Hermitian random matrices $W$.
We have $ W \in \R^{n\times n}$ or $ W\in\C^{n\times n}$, respectively, and its entries are independent random variables except for being Hermitian, $W_{ij} = \overline{W_{ji}}$. The off-diagonal entries are real or complex standard Gaussian random variables,\footnote{A complex standard Gaussian has the law of $x + iy$ where $x, y \sim \mathcal{N}(0, \frac{1}{2})$ are independent.} and the diagonal entries follow real Gaussian distribution with $W_{ii} \sim \mathcal{N}(0, 2)$ or $W_{ii} \sim \mathcal{N}(0, 1)$, respectively.
We write $\GOE(n)$ and $\GUE(n)$ for the respective laws.
\end{definition}

\begin{definition}[Multi-frequency $\bb S$-synchronization]\label{def:angular_synch_U1}
Let $x \in \C^n$ have every coordinate sampled independently from uniform distribution over $U(1)$. %
Fix a number of frequencies $L$ and let $\lambda_\ell \ge 0$ be the signal-to-noise ratio at frequency $\ell \in [L]$. We observe $L$ matrix observations as follows:
    \begin{equation}\label{eq:gaussian_synch_unitary}
        Y_\ell = \frac{\lambda_\ell}{n} x^{(\ell)} x^{(\ell)*} + \frac{1}{\sqrt{n}} W_\ell, \quad \text{for $\ell = 1, \ldots L$},
    \end{equation}
    where $W_{\ell} \sim \GUE(n)$ are drawn independently. By $x^{(\ell)}$ we denote entrywise $\ell$-th power
    $$
    x^{(\ell)} = (x_1^\ell, x_2^\ell, \dots, x_n^\ell)^{\top}.
    $$
    We then denote the distribution of $(Y_1, \dots, Y_L)$ by $\GSynch(\bb S, L, \lambda)$.
\end{definition}
\noindent
We will compare the planted distribution $\bb P_n \colonequals \GSynch(\bb S, L, \lambda)$ for $\lambda = (\lambda_1, \dots, \lambda_L)$ with the corresponding pure noise model $\bb Q_n \colonequals \GSynch(\bb S, L, (0, \dots, 0))$.

In \cite{Perry_AMP_synch}, it was predicted that the detection in the synchronization model with multiple frequencies is not efficient below the spectral threshold $\lambda \le 1$ (assuming all $\lambda_\ell$ to be the same). We confirm this prediction in sense of the low-degree hardness. 
\begin{theorem}[Formalized version of \Cref{thm:angular_comp_tr_intro}]
    Consider the Gaussian angular synchronization model with $L$ frequencies where $L$ does not depend on $n$. 
    Denote by $\lambda_{\text{max}}$ the maximum signal-to-noise ratio value among all the frequencies, 
    $$
    \lambda_{\text{max}} = \max_{\ell=1, \dots, L} \lambda_\ell.
    $$
If $\lambda_{\text{max}} \le 1$ and $D = D(n) = o(n^{1/3})$,
it holds $\| L^{\le D}_n\| = O(1)$ as $n \to \infty$.
\end{theorem}

\begin{remark}
 In particular, assuming Conjecture~\ref{conj:low-degree}, this theorem suggests the failure of all polynomial-time algorithms. 
\end{remark}
\noindent
We highlight that the theorem implies a sharp computational phase transition. Indeed, once $\lambda_{\text{max}} \ge 1$, or equivalently, there is at least one frequency $\ell$ such that $\lambda_\ell \ge 1$, one can use PCA on the corresponding observation to detect the presence of the signal. 

We note also that for our low-degree calculation we can make the simplifying assumption that all of the  $\lambda_\ell$ are the same. If this is not the case, we can set all $\lambda_{\ell} \colonequals \lambda_{\text{max}}$. It can be easily verified that the second moment of the LDLR only increases in this case. Thus, low-degree hardness of the modified model implies the low-degree hardness of the original model. This assumption simplifies the analysis substantially.

The main idea of the proof is to show a connection to another multi-frequency angular synchronization model, this time with a different prior distribution on $x$. We will define the model with $L$-cyclic prior, in which every coordinate is sampled from the roots of unity as opposed to the complex unit circle. This variant can be seen as synchronization over $L$-cyclic group of integers modulo $L$.

\begin{definition}[Multi-frequency $\bb Z_L$-synchronization]\label{def:angular_synch_ZL}
    Fix $L \ge 2$ and let $x \in \C^n$ be such that every coordinate is sampled independently as $x_j \sim \Unif(\{\omega_0, \dots, \omega_{L-1}\})$ where $\omega_k = \exp(2\pi i k / L)$. Let $\lambda \ge 0$. We consider $\ceil{L/2} - 1$ matrix observations:
    \begin{equation*}%
        Y_\ell = \frac{\lambda}{n} x^{(\ell)} x^{(\ell)*} + \frac{1}{\sqrt{n}} W_\ell, \quad \text{for $\ell = 1, \ldots \ceil{L/2} - 1$},
    \end{equation*}
    where are drawn independently.
    For all $\ell \ne L/2$ we take $W_{\ell} \sim \GUE(n)$, and if $L$ is even, then we take $W_{L/2} \sim \GOE(n)$.
    We then denote the distribution of $(Y_1, \dots, Y_L)$ by $\GSynch(\bb Z_L, \lambda)$.
\end{definition}

    \begin{remark}
        The number of frequencies in the above definition is chosen to avoid redundancy in the set of observations: we exclude the trivial frequency (the frequency corresponding to $L$-th power), which would not depend on $x$, and take only one frequency per conjugate pair. 
    \end{remark}  %

We will proceed in two steps to show the computational hardness of these two models.
First, %
as mentioned above, we will show in \Cref{sec:ang_cyclic_reduction} that detection by low-degree polynomials in the angular synchronization model with $L$ frequencies is at least as hard as detection in its $L$-cyclic counterpart. 
The basis for this reduction is \Cref{lm:reduction_angular_cyclic}, which shows that in the latter case, the second moment of the low-degree likelihood ratio is larger. The proof of hardness for $\GSynch(\bb S, L)$ is a special case of a more general result for synchronization models over arbitrary finite groups which is described in \Cref{sec:finite_groups}. However, for the purposes of illustration, we present the beginning of the argument in \Cref{sec:l-cyclic} in this special case as it does not require any knowledge of representation theory. The final part of the proof is a combinatorial analysis that is the same as in the general case, and can be found in \Cref{sec:combinatorics_proof_finite}. We provide a short sketch of this part of the proof in \Cref{sec:sketch_combinatorial_proof}.

\section{Synchronization over the circle group: proofs}

\subsection{Bounding $\GSynch(\bb S, L, \lambda)$ by $\GSynch(\bb Z_L, \lambda)$}\label{sec:ang_cyclic_reduction}

Our first goal will be to show the following comparison result.
\begin{lemma} \label{lm:reduction_angular_cyclic}
    Denote by $L_{n, \bb S}^{\le D}$ the low-degree likelihood ratio for detection in the model with the uniform angular prior, $\GSynch(\bb S, L, \lambda)$, and $L_{n, \bb Z_L}^{\le D}$ for $L$-cyclic prior, $\GSynch(\bb Z_L, \lambda)$. For any $L, D, n \ge 1$, we have
    $$
    \| L_{n, \bb S}^{\le D}\|^2 \le \| L_{n, \bb Z_L}^{\le D}\|^2.
    $$
\end{lemma}

To prove the lemma we will need several intermediate steps. We start with deriving an expression for the second moment of the low-degree likelihood ratio. Despite observing multiple observation matrices, both models are Gaussian additive models, as all observations can be stacked as a single vector containing the concatenated signals with Gaussian noise added. We will use this fact to calculate closed-form expressions for $\| L_{n, \bb S}^{\le D}\|^2$ and $\| L_{n, \bb Z_L}^{\le D}\|^2$ by utilizing previous results for the LDLR for general Gaussian additive models. 

\begin{lemma}[LDLR of the Gaussian additive model]
\label{lm:ldlr_gaussian_additive_model}
Let $\theta$ be a $N$-dimensional vector defined either over $\R$ or $\C$ to be drawn from some prior distribution $\mathcal P_N$. Let $Z$ be a random vector of dimension $N$ with independent Gaussian entries of the respective type (as in \Cref{def:gaussian_type}). We define $\bb P_N$ and $\bb Q_N$ as follows.
\begin{itemize}
    \item Under $\bb P_N$, observe $Y = \theta + Z$ (planted distribution).
    \item Under $\bb Q_N$, observe $Y = Z$ (null distribution).
\end{itemize}
    Set $\beta = 1$ for $\R$ and $\beta = 2$ for $\C$. Then the norm of the low-degree likelihood ratio between $\bb P_N$ and $\bb Q_N$ is 
    \begin{equation}\label{eq:ldlr_GAM}
        \|L_n^{\le D}(Y)\|^2 = \E_{\theta, \theta^\prime \sim P_N} \lt [\sum_{d=0}^D \frac{1}{d!}  \bigg( \beta \Re(\langle \theta, \theta^\prime \rangle)\bigg)^d \rt],
    \end{equation}
    where $\theta$ and $\theta^\prime$ are drawn independently from $\mathcal{P}_N$.
\end{lemma}
\begin{proof} 
For the real case the expression coincides with the one given in \cite[Theorem~2.6]{Kunisky2019NotesLowDeg}. For the complex case, we essentially reduce the model to the real Gaussian additive model by considering real and imaginary parts of the observation separately.

Consider such $2N$-dimensional observation vector $Y = (Y_{\Re}, Y_{\Im})\in\R^{2N}$ as follows:
\begin{align*}
    Y_{\Re} = \sqrt{2} \Re(\theta) + Z_{\Re} \\
    Y_{\Im} = \sqrt{2} \Im(\theta) + Z_{\Im},
\end{align*}
where $Z_{\Re}, Z_{\Im} \in \R^N$ are real-valued noise vectors with independent Gaussian entries. The constant $ \sqrt{2}$ before the signal comes from the fact that the real and imaginary parts of the original complex noise vector have variance~$1/2$. To match the unit variance in the real case, we multiply the observation by $ \sqrt{2}$.

For this decomposed observation the low-degree likelihood ratio is written as 
\begin{align*}
\|L_n^{\le D}(Y)\|^2  &= \E_{\theta, \theta^\prime \sim P_N} \lt [\sum_{d=0}^D \frac{1}{d!}    \lt \langle \sqrt{2} \begin{pmatrix}
    \Re(\theta) \\ \Im(\theta)
\end{pmatrix}, \sqrt{2}\begin{pmatrix}
    \Re(\theta^\prime)\\ \Im(\theta^\prime)
\end{pmatrix} \rt\rangle^d \rt]\\
&= \E_{\theta, \theta^\prime \sim P_N} \lt [\sum_{d=0}^D \frac{1}{d!}  2^d \lt (\langle \Re(\theta),  \Re(\theta^\prime)\rangle + \langle \Im(\theta), \Im(\theta^\prime) \rangle\rt)^d \rt] \\
&= \E_{\theta, \theta^\prime \sim P_N} \lt [\sum_{d=0}^D \frac{1}{d!} \lt(2\Re \langle \theta, \theta^\prime \rangle\rt)^d \rt].
\end{align*}
This concludes the proof.
\end{proof}
With this result, we easily derive the second moment of LDLR for the considered synchronization models.
\begin{lemma}\label{lm:ldlr_angular_cyclic_base_expression} For the angular synchronization model $\GSynch(\bb S, L, \lambda)$, the LDLR second moment is expressed as
    $$
    \|L_{n, \bb S}^{\le D}\|^2 = \sum_{d=0}^D \frac{1}{d!} \E\lt( \frac{\lambda^2}{n}\sum_{\ell=1}^L \ip{x^{(\ell)}, \bm 1_n }^2 \rt)^d,
    $$
    where $\bm 1_n \in \R^n$ is a vector of all ones. 
    For the synchronization model $\GSynch(\bb Z_L, \lambda)$, we can write LDLR as
    $$
    \|L_{n, \bb Z_L}^{\le D}\|^2 = \sum_{d=0}^D \frac{1}{d!} \E\lt( \frac{\lambda^2}{n}\sum_{\ell=1}^L \frac{\beta_\ell }{2}\ip{x^{(\ell)}, \bm 1_n }^2 \rt)^d,
    $$
    where $\beta_\ell = 1$ only for $\ell = L/2$ when $L$ is even and $\beta_\ell = 2$ otherwise.
\end{lemma}
\begin{proof} The only difference between considered model and the model in \Cref{lm:ldlr_gaussian_additive_model} is that the noise matrix is symmetric. %
Nevertheless, the model is equivalent to observing a signal perturbed by a Gaussian matrix with independent entries with the scaled variance to match the original SNR. The new asymmetric model is then 
$$
\tilde{Y}_\ell = \frac{\lambda}{n} X_\ell X_\ell^* + \sqrt{\frac{2}{n }} \tilde{W}_\ell,
$$
where $\tilde{W}_\ell$ is a random matrix, whose entries are independent Gaussian random variables, either complex or real depending on the type of the frequency. We scale the variance in $\sqrt{2}$ times to match the original SNR. 
This equivalence can be also verified by writing entrywise density of $Y_{kj}$. We omit these details and refer to \cite[Appendix~A.2]{Kunisky2019NotesLowDeg} for a more detailed explanation. 

It is easy to see that now both angular synchronization models are Gaussian additive models as in \Cref{lm:ldlr_gaussian_additive_model} with the signal $\theta \in \C^{Ln^2}$ which consists of concatenated flattened matrices $\frac{\lambda}{\sqrt{2n}}x^{(\ell)} x^{(\ell)*} $ for $\ell = 1, \dots, L$.

    We start first with the uniform $U(1)$ case since all frequencies are complex. The desired expression follows from simple algebraic manipulations:
    $$
    \ip{\theta, \theta^\prime} = \sum_{\ell=1}^L \frac{\lambda^2}{2n} \tr\lt(x^{(\ell)} x^{(\ell)*} (x^\prime)^{(\ell)} (x^\prime)^{(\ell)*}  \rt) = \frac{\lambda^2}{2n} \sum_{\ell=1}^L \ip{x^{(\ell)}, (x^\prime)^{(\ell)}}^2.
    $$
    Since the prior is symmetric under rotations, the random variables $\ip{x^{(\ell)}, (x^\prime)^{(\ell)}}$ and $\ip{x^{(\ell)}, \bm 1_n}$ are identically distributed, and thus we can replace $\ip{x^{(\ell)}, (x^\prime)^{(\ell)}}$ by $\ip{x^{(\ell)}, \bm 1_n}$.
    
    By substituting the resulting expression into \eqref{eq:ldlr_GAM} we complete the proof for the uniform angular case. The same proof applies for $L$-cyclic prior with odd $L$.

    For $L$-cyclic prior with even $L$, we have to account for different type of frequencies. We do it in the same way as in the proof of \Cref{lm:ldlr_gaussian_additive_model} by reducing the model to the real Gaussian model. The signal $\theta$ consists of scaled real and imaginary parts for complex frequencies 
    $$
    \frac{\lambda}{\sqrt{2n}} \sqrt{2}\Re \lt(x^{(\ell)} x^{(\ell)*}\rt), \quad \frac{\lambda}{\sqrt{2n}} \sqrt{2}\Im \lt(x^{(\ell)} x^{(\ell)*}\rt)
    $$
    and the unchanged signal part for $\ell=L/2$ 
    $$
    \frac{\lambda}{\sqrt{2n}}x^{(L/2)} x^{(L/2)*}.
    $$
    Similarly as for $\bb S$ case, we obtain,
    $$
    \ip{\theta, \theta^\prime} = \sum_{\ell=1}^L \frac{\lambda^2}{2n} \beta_\ell \ip{x^{(\ell)}, (x^\prime)^{(\ell)}}^2 =  \frac{\lambda^2}{n} \sum_{\ell=1}^L \frac{\beta_\ell}{2} \ip{x^{(\ell)}, (x^\prime)^{(\ell)}}^2.
    $$
    Replacing $\ip{x^{(\ell)}, (x^\prime)^{(\ell)}}$ by $\ip{x^{(\ell)}, \bm 1_n}$ and substituting the resulting expression into \eqref{eq:ldlr_GAM} completes the proof.
\end{proof}

Recall that the second moment of the LDLR involves the moments of the random variable $\sum_{\ell=1}^L | \langle x^{(\ell)}, \bm 1_n \rangle |^2 $. We will rewrite this quantity as cardinality of a certain set to show the connection between two models. It will be easy to see that the set for the $L$-cyclic prior is contained in the corresponding set for the uniform $U(1)$ prior which will form the basis for the proof of \Cref{lm:reduction_angular_cyclic}. 
\begin{lemma}\label{lm:expect_as_cardinality}
    Let $x \in \bb C^n$ be sampled either from the uniform angular prior or $L$-cyclic prior. Then we have
    $$
    \E_{x} \lt( \sum_{\ell=1}^L | \langle x^{(\ell)}, \bm 1_n \rangle |^2 \rt)^d = \lvert M_d \rvert,
    $$
    where for the angular prior
    $$
    M_d = M_d^{\bb S} \colonequals \Big \{ \bm{\ell} \in [L]^d, \mathbf{a}, \mathbf{b} \in [n]^d : \sum_{j = 1}^d \ell_j (\mathbf{e}_{a_j} - \mathbf{e}_{b_j}) = 0 \Big \}
    $$
    and for the cyclic prior, %
    $$
    M_d = M_d^{\bb Z_L} \colonequals \Big \{ \bm{\ell} \in [L]^d, \mathbf{a}, \mathbf{b} \in [n]^d : \sum_{j = 1}^d \ell_j (\mathbf{e}_{a_j} - \mathbf{e}_{b_j}) = 0 \pmod L \Big \}.
    $$
\end{lemma}

\begin{proof}%
    Using the angle representation as in the beginning of the section, we rewrite the moment as
\begin{equation*}
    \begin{split}
        &\E_{x} \Big( \sum_{\ell=1}^L | \langle x^{(\ell)},1 \rangle |^2 \Big)^d = \E_{x}\Big(   \sum_{\ell=1}^L \big | \sumjn e^{i \ell \varphi_j(x)} \big |^2 \Big)^d,
    \end{split}
\end{equation*}
where $\varphi_j(x)$ is a phase of coordinate $x_j = e^{i \varphi_j}$.

Expanding the scalar product and power we get
\begin{align}
        &\E_{\varphi_1, \dots \varphi_n} \Big(   \sum_{\ell=1}^L \big | \sumjn e^{i \ell \varphi_j} \big |^2 \Big)^d = \E \Big ( \sum_{\ell=1}^L \sum_{a, b =1}^n e^{i \ell (\varphi_a - \varphi_b)} \Big)^d  \nonumber\\
        &\quad =\E \sum_{\ell_1, \dots, \ell_d = 1}^L \sum_{\substack{a_1, \dots, a_d = 1\\b_1, \dots, b_d = 1}}^n \exp{\big[ i (\ell_1 (\varphi_{a_1} - \varphi_{b_1}) + \dots + \ell_d (\varphi_{a_d} - \varphi_{b_d}) \big]} \nonumber\\
        &\quad =\E \sum_{\ell_1, \dots, \ell_d = 1}^L \sum_{\substack{a_1, \dots, a_d = 1\\b_1, \dots, b_d = 1}}^n \exp{\lt[ i \sum_{s=1}^n  \lt( \varphi_s\sum\limits_{\substack{j, k: a_j = s,\\ b_k = s}} (\ell_j - \ell_k)  \rt) \rt]} \label{eq:power_d_collecting_phi}\\
        &\quad= \sum_{\ell_1, \dots, \ell_d = 1}^L \sum_{\substack{a_1, \dots, a_d = 1\\b_1, \dots, b_d = 1}}^n \prod_{s=1}^n \E \exp{\big[ i K_{s} \varphi_s \big]},\label{eq:expect_powerd}
\end{align}
where $K_s = \sum\limits_{\substack{j, k: a_j = s,\\ b_k = s}} (\ell_j - \ell_k)$. In \eqref{eq:power_d_collecting_phi}, we collect coefficients of $\varphi_s$ with the same index $a_j = s$ and $b_k = s$. In the last line \eqref{eq:expect_powerd}, we used independence between the coordinates of the signal. 

Note that for the uniform $U(1)$ prior, it holds $\E e^{i K_s \varphi_s } = \mathds{1}(\{K_s = 0\})$. For the $L$-cyclic prior, $\E e^{i K_s \varphi_s } = \mathds{1}(\{K_s = 0 \mod L\})$. Consequently, the expression in \eqref{eq:expect_powerd} equals the number of terms where the exponent expression is zero. In other words, it equals the cardinality of the set
$$
\Big \{ \bm{\ell} \in [L]^d, \mathbf{a}, \mathbf{b} \in [n]^d : \sum_{j = 1}^d \ell_j (\mathbf{e}_{a_j} - \mathbf{e}_{b_j}) = 0 \Big \},
$$
where $\mathbf{e}_j, j = 1, \dots, n$  are standard basis vectors in $\bb R^n$. For the cyclic prior, the equality in the above display is modulo $L$.
\end{proof}
\noindent
With this result, the proof of \Cref{lm:reduction_angular_cyclic} is straightforward.
\begin{proof}[Proof of \Cref{lm:reduction_angular_cyclic}]
Note that for the sets defined in \Cref{lm:expect_as_cardinality}, it holds
$$
M_d^{\bb S} \subseteq M_d^{\bb Z_L},
$$
for every $d\ge 0$. Hence by \Cref{lm:expect_as_cardinality}, we have 
\begin{align*}
    \| L^{\le D}_{n, \bb S} \|^2 &= \sum_{d=0}^D \frac{1}{d!}  \E_{x} \Big ( \frac{\lambda^{2}}{n} \sum_{\ell=1}^L \langle x^{(\ell)} , 1 \rangle^2 \Big)^d \\
 &= \sum_{d=0}^D \frac{1}{d!} \lt |M_d^{\bb S}\rt| \\
 &\le \sum_{d=0}^D \frac{1}{d!} \lt |M_d^{\bb Z_L}\rt |  =\| L^{\le D}_{n, \bb Z_L} \|^2, 
\end{align*}
completing the proof.
\end{proof}
\noindent
Thus when we bound $\| L^{\le D}_{n, \bb Z_L} \|^2 $ we immediately obtain also a bound on $\| L^{\le D}_{n, \bb S} \|^2$. 

\subsection{Multinomial expression of $\| L^{\le D}_n \|^2$ for $\GSynch(\bb Z_L, \lambda)$}\label{sec:l-cyclic}
In this section, we provide the first part of the argument for the hardness of $\GSynch(\bb Z_L, \lambda)$. The idea of this part is to rewrite the second moment of the ratio using integer random variables. These variables count how many times every root of unity $\omega_\ell$ was sampled in the signal vector $x$.

\begin{lemma}\label{lm:ldlr_counts_ng_cyclic} %

Let $n_0, \dots, n_{L-1}$ be distributed according to the multinomial distribution with probabilities $1 / L$ and number of trials $n$ (that is, the numbers of balls in each of $L$ bins after $n$ balls are thrown independently each into a uniformly random bin).
The second moment of the LDLR can be rewritten as following. 
\begin{equation}\label{eq:ldlr_counts_nL_cyclic}
    \| L^{\le D}_n \|^2 = \sum_{d=0}^D \frac{1}{d!}  \frac{\lambda^{2d}}{n^d} \E \left(  \frac{L-1}{2}\sum_{\ell = 0}^{L-1} n_\ell^2 - \frac{1}{2} \sum_{\substack{\ell, k = 0\\ \ell\ne k}}^{L-1} n_\ell n_k \right)^d ,
\end{equation}
\end{lemma}

The argument and the resulting expression is the same for all finite groups of size $L$ (see \Cref{lm:ldlr_counts_ng}), but we present it for this special case for clarity and ease of understanding, as the broader argument relies on some notions of representation theory. The remainder of the proof is universal for all groups and can be found in \Cref{sec:combinatorics_proof_finite}.

\begin{proof}
    In this proof, we will focus only on the case of even $L$, as the other case is analogous. 
For $x \in \C^n$ supported only on roots of unity $\omega_0, \dots, \omega_{L-1}$, define $n_\ell(x) \colonequals |\{j: x_j = \omega_\ell  \}|\in \bb N$.
    Recall from \Cref{lm:ldlr_angular_cyclic_base_expression},
    $$
    \|L_{n, \bb Z_L}^{\le D}\|^2 = \sum_{d=0}^D \frac{1}{d!} \E \lt( \frac{\lambda^2}{n}\sum_{\ell=1}^{L/2} \frac{\beta_\ell }{2}\ip{x^{(\ell)}, 1 }^2 \rt)^d.
    $$
    With the introduced counts $n_\ell(x)$, we can write $\langle x, 1 \rangle = n_0 + n_1 e^{2 \pi i / L} + \dots + n_{L-1} e^{2 \pi i (L - 1) / L}$. 
    Note that if $x$ is sampled from the $L$-cyclic prior, $n_0, \dots, n_{L-1}$ follow multinomial distribution. Hence, the second moment of the likelihood ratio can be expressed as follows:
\begin{equation*}
\begin{split}
    \| L^{\le D}_n \|^2 = \sum_{d=0}^D \frac{\lambda^{2d}}{n^d d!}  \E \bigg ( &\sum_{\ell = 1}^{L/2-1} \lt | n_0 + n_1 e^{2 \pi i \ell / L } + \dots + n_{L-1} e^{2 \pi i \ell (L - 1)  / L} \rt |^{2}  \\
    &\quad+ \frac 1 2  \cdot (n_0 - n_1 + \dots + n_{L-2} - n_{L-1})^2 \bigg )^d,
\end{split}
\end{equation*}
where the expectation is now taken with respect to $n_0, \dots, n_{L-1}$.

    We recognize the discrete Fourier transform of $n_\ell$ in the inner sum. By denoting the Fourier coefficients as $\hat{n}_k = \sum_{\ell=0}^{L-1} e^{2\pi i \ell k / L} x_\ell$ and completing the sum, we have
    $$
    \| L^{\le D}_n \|^2 = \sum_{d=0}^D \frac{\lambda^{2d}}{n^d d!} \lt ( \frac{1}{2}\sum_{\ell=0}^{L-1} |\hat{n}_\ell |^2 - \frac{1}{2} |\hat{n}_0|^2 \rt)^d.
    $$
    By Parseval's theorem, 
    $$
    \sum_{\ell=0}^{L-1} n_\ell^2 = \frac{1}{L} \sum_{k=0}^{L-1} |\hat{n}_k|^2,
    $$
    and hence, 
    \begin{align*}
    \| L^{\le D}_n \|^2 &= \sum_{d=0}^D \frac{\lambda^{2d}}{n^d d!} \lt ( \frac{L}{2}\sum_{\ell=0}^{L-1} |n_\ell |^2 - \frac{1}{2} \lt ( \sum_{\ell=0}^{L-1} n_\ell \rt)^2 \rt)^d \\
    &= \sum_{d=0}^D \frac{\lambda^{2d}}{n^d d!} \lt ( \frac{L-1}{2}\sum_{\ell=0}^{L-1} n_\ell ^2 - \frac{1}{2}  \sum_{\substack{\ell,k=0\\ \ell\ne k}}^{L-1} n_\ell n_k  \rt)^d,
    \end{align*}
    completing the proof.
\end{proof}

\subsection{Sketch of combinatorial analysis of $\GSynch(\bb Z_L, \lambda)$}\label{sec:sketch_combinatorial_proof}

Finally, we very briefly outline the remainder of the combinatorial arguments that we will use for bounding this quantity.
The main observation is to view the quantity inside the expectation as a quadratic form on the vector of $n_i$'s:
\begin{equation}
\frac{L-1}{2}\sum_{\ell = 0}^{L-1} n_\ell^2 - \frac{1}{2} \sum_{\substack{\ell, k = 0\\ \ell\ne k}}^{L-1} n_\ell n_k = (n_0, \dots, n_{L-1})^{\top} \lt(\frac{L}{2} I_{L\times L} - \frac{1}{2} \bm 1_L \bm 1_L^{\top}\rt) (n_0, \dots, n_{L-1}).
\end{equation}
We note that the the matrix involved is positive semidefinite, and the all-ones vector $\bm 1_L$ is an eigenvector of this matrix with eigenvalue zero.
All other eigenvectors $v$ will be orthogonal to $\bm 1_{L}$ and have positive eigenvalue $\lambda > 0$.
Our main idea will be to consider the spectral decomposition of this matrix, which will lead to a sum of terms of the form $\lambda (\sum_{i = 0}^{L - 1} v_i n_i)^2$ where $\sum v_i = 0$.
The linear expressions inside will be \emph{centered} random variables, and we will be able to control their moments by a careful combinatorial expansion.

An analogous task will arise in the analysis of arbitrary finite groups $G$, which we show below before giving the details of this analysis.
We also give further details for the special case $L = |G| = 3$ as an illustrative example in Section~\ref{sec:finite_3_group}.

\section{Synchronization over finite groups: main results}\label{sec:finite_groups}

\subsection{Short preliminaries on representation theory}\label{sec:representation_theory}
To define synchronization model over a compact group, we need to introduce several concepts of group and representation theory. We assume the reader to be familiar with the basics of representation theory and revisit only a few key concepts for the completeness of the exposition. For a more in-depth review on the subject we refer the reader to~\cite{brocker2013representations}.

\paragraph{Normalized Haar measure} We define signal prior using the normalized Haar measure which can be viewed as an analog of uniform distribution over a group. On a compact group $G$, the Haar measure is the unique left-invariant measure, moreover, it is also invariant to the right translation by any group element and is finite, $\mu(G) < \infty$.
Informally speaking, the Haar measure assigns an invariant volume to subsets of a group. For instance, on finite groups, every Haar measure is a multiple of counting measure, and the normalized Haar measure accounts to the uniform distribution in a classical sense. On compact groups, the normalized Haar measure defines the unique invariant Radon probability measure. We will use this measure as the prior of the signal components and we will understand integrals of the form 
$$
\int_G f(g) \rm d g
$$
to be taken with respect to the Haar measure. %

\paragraph{Irreducible representations}
For $\bb S$-synchronization, the multi-frequency model \eqref{eq:gaussian_synch_unitary} can be seen as observing the signal through noisy channels corresponding to different Fourier modes. This model can be motivated by taking Fourier transform of non-linear observation function of relative alignments, see for the details~\cite{Bandeira_NonUniqueGames_2020}. In case of general compact groups, the expansion is given in terms of the irreducible representations as described by the Peter-Weyl theorem. 
To highlight the connection with the angular synchronization, observe that $U(1) $ is an irreducible representation of $SO(2)$, and the set of the roots of unity $\{\omega_\ell\}_{\ell=1}^L = \{e^{2 \pi i \ell / L}\}_{\ell=1}^L$ is an irreducible representation of $\bb Z_L$. 

We note a few important properties of the irreducible representations. First, if $G$ is a finite group, it has only a finite number of irreducible representations. Secondly, for a nontrivial irreducible representation $\rho$ we have $\int_G \rho(g) d g = 0$.

\paragraph{Peter-Weyl decomposition} %
The representation-theoretic analog of the Fourier series is given by Peter-Weyl decomposition. 
The Peter-Weyl theorem provides an orthonormal basis with respect to the Hermitian inner product defined on $L^2(G)$ as following 
$$
\langle f, h\rangle = \int_G  f(g) \bar h(g) \rm d g.
$$
This orthonormal basis is formed with the coefficients of scaled irreducible representations, $\sqrt{d_\rho} \rho_{ij}$ indexed over complex irreducible representation $\rho$ of $G$ and indices $1 \le i, j \le d_\rho$, where $d_\rho$ is the dimension of $\rho$.

For a given function $f : G \to \R $, its Peter-Weyl decomposition can be written in terms of irreducible representations,
$$
f(g) = \sum_\rho \sqrt{d_\rho} \langle \hat f(g)_\rho , \rho(g)\rangle,
$$
where the coefficients $\hat {f}_\rho $ are matrices of dimension of the respective representation $\rho$. Similarly to Fourier transform, these coefficients can be computed from $f(g)$ by integration:
$$
\hat f_\rho = \int_G \sqrt{d_\rho}\rho(g)f(g) \rm d g.
$$

For the cyclic group of size $L$, the irreducible representations are $\{e^{2\pi i k / L}\}_{k=1}^L$, and Peter-Weyl decomposition reduces to discrete Fourier transform 
$$
f(g) =\frac{1}{L} \sum_{k=1}^L \hat f(g)_k e^{2 \pi i k / L}.
$$

\paragraph{Representations of real, complex, quaternionic type}
Every irreducible complex representation has one of three types: real, complex, or quaternionic. In the considered problem, we have to slightly adjust the model and the analysis depending on the type.

A complex representation $\rho$ is of \emph{real type} if it is isomorphic to a real-valued representation. Without loss of generality, we can assume that $\rho$ is a real-valued scalar or matrix with real-valued entries. 

We say that representation $\rho$ is of \emph{complex type} if it is not isomorphic to its complex conjugate $\bar \rho(g) = \overline{\rho(g)}$. %

Finally, representation $\rho$ is of \emph{quaternionic type} if it can be defined over quaternion field $\bb H$, i.e., matrix with a quaternion-valued entries $a + b i + c j + d k $, where $1, i, j, k$ are the basis vectors. We will refer to the first coefficient as the real part, $\Re(a + b i + c j + d k) = a$.

For our setting it will be more convenient to define these representation over complex numbers. %
Each quaternion $a + b i + c j + d k$ can be equivalently represented as $2\times 2$ block of the form
$$
\begin{pmatrix}
    a + b i & c  + d i \\
    -c + d i & a - b i
\end{pmatrix}.
$$
The expression for the real part becomes the real part of the first  (or, equivalently, second) diagonal element. The squared norm of quaternion is defined by the sum of squared coefficients or half the Frobenius norm:
$$
\|a + bi + cj +dk \|^2 = \frac{1}{2} \lt \| \begin{pmatrix}
    a + b i & c  + d i \\
    -c + d i & a - b i
\end{pmatrix}\rt \|_F^2 = a^2 + b^2 + c^2 + d^2. 
$$
The conjugate is defined as $(a + b i + c j + d k)^* = a - b i - c j - d k$, or, equivalently, as a conjugate transpose of the respective complex matrix. For two quaternionic vectors $x_1, x_2 \in \bb H^d $ ($x_1, x_2 \in \C^{2d \times 2}$) the scalar product is defined as 
$$
\langle x_1,  x_2 \rangle = \sum_{\ell=1}^d (x_1)_{\ell} (x_2)_{\ell}^*,
$$
where $(x_1)_{\ell}$ denotes $\ell$-th element of $x_1$. 
For the analysis, we will need the expression for the real part of such scalar product which for a single element can be written as
\begin{align*}
&\Re \lt (\langle a_1 + b_1 i + c_1 j + d_1 k , a_2 + b_2 i + c_2 j + d_2 k \rangle \rt) \\
&\quad= a_1 a_2 + b_1 b_2 + c_1 c_2 + d_1 d_2 \\
&\quad= \frac{1}{2} \Re \tr\lt ( \begin{pmatrix}
    a_1 + b_1 i & c_1  + d_1 i \\
    -c_1 + d_1 i & a_1 - b_1 i
\end{pmatrix} \begin{pmatrix}
    a_2 + b_2 i & c_2  + d_2 i \\
    -c_2 + d_2 i & a_2 - b_2 i
\end{pmatrix}^* \rt).
\end{align*}

\subsection{Model definition and formal statement of results}

To account for the different types of the representation, we adjust the noise model depending on the type of the frequency. For the definitions of Gaussian orthogonal and unitary ensembles, refer to \Cref{def:goe}. Here we define one remaining type of the ensemble, namely, Gaussian symplectic ensemble.  %

\begin{definition}\label{def:gaussian_type}
    We say that a random variable $w$ is a standard Gaussian variable
       of \emph{quaternionic} type if $w \in \C^{2\times 2}$ and $w$ encodes a quaternion whose four values are drawn from $\mathcal{N}(0, 1/4)$, i.e., 
        $$
        w = \begin{pmatrix}
    a + b i & c  + d i \\
    -c + d i & a - b i
\end{pmatrix},
        $$
        where $a, b, c, d \sim \mathcal{N}(0, 1/4)$ and are independent. 
\end{definition}

\begin{definition}\label{def:gse} %
Let $W \in \C^{2n \times 2n}$ be a random Hermitian matrix. We say that $W$ is drawn from Gaussian symplectic ensemble (GSE) if the following conditions hold. The $2\times 2$ blocks of $W$ decode Gaussian variables of quaternionic type and they are independent except for the symmetry. The off-diagonal blocks are standard Gaussian random variables of quaternionic type, and the diagonal blocks decode the real Gaussian variables with $\mathcal{N}(0, 1/2)$ in quaternionic form.  
\end{definition}
See also \Cref{table:GOE-GUE-GSE} for the summary of Gaussian ensembles definitions. 

Observe that each of the defined ensembles can also be seen as averaging a matrix with independent Gaussian variables (of the respective type) and its conjugate transpose. This perspective explains having real Gaussians on the diagonal and the difference in variance between off-diagonal and diagonal entries. We choose this model because the signal counterpart in the observation is Hermitian itself. This means that we observe each off-diagonal element $(k, j), k \ne j$ twice: $(Y_{\rho})_{kj}$ and its conjugate. If we had different noise on these elements, we could easily reduce the variance by taking an average of $(Y_{\rho})_{kj}$ and $\overline{(Y_{\rho})_{jk}}$.

\begin{table}%
\begin{center}
\begin{tabular}{ |c|c|m{7.5cm}|m{4.5cm}| } 
 \hline
 Ensemble & Dimension & Off-diagonal entries  & Diagonal entries \\ 
 \hline \hline
 GOE & $\R^{n \times n}$ & $\mathcal{N}(0, 1)$ & $\mathcal{N}(0, 2)$ \\ 
 GUE & $\C^{n \times n}$ & $\mathcal{N}(0, 1/2) + i \mathcal{N}(0, 1/2 )$ & $\mathcal{N}(0, 1)$ \\ 
 GSE & $\bb C^{2 n \times 2 n}$ &  $\begin{pmatrix}
    a + b i & c  + d i \\
    -c + d i & a - b i
\end{pmatrix}$  where $a, b, c, d \sim \mathcal{N}(0, 1/4)$ &  %
$\begin{pmatrix}
    a & 0 \\
   0 & a 
\end{pmatrix}$  where $a \sim \mathcal{N}(0, 1/2) $\\
 \hline
\end{tabular}
\end{center}
\caption{Summary of \Cref{def:goe} and \Cref{def:gse}}
\label{table:GOE-GUE-GSE}
\end{table}

For the signal prior, we assume that all coordinates of the signal are independent and identically distributed (iid). We fix group $G$ and choose the normalized Haar measure over $G$ as the coordinate prior distribution. As it was mentioned in \Cref{sec:representation_theory}, this measure is an analog of uniform distribution for the groups. In particular, in case of finite groups, this means that each element of the group is picked with equal probability.

We are now fully equipped to state the Gaussian synchronization model over a compact group. 
\begin{definition}[{{Adapted from~\cite[Definition~6.7]{perryOptimalitySuboptimalityPCA2016}}}] \label{def:GSM} Let $G$ be a compact group and let $\Psi$ be a set of all non-isomorphic irreducible representations of $G$ (excluding trivial representation and taking only one element from pair $\rho$ and its conjugate $\bar \rho$). 

The Gaussian synchronization model is defined as follows. %
Draw a vector $u \in G^n$ by sampling independently each coordinate from Haar (uniform) measure on $G$. %
For each irreducible representation $\rho \in \Psi$ we obtain a matrix observation depending on the type of representation as follows.
\begin{itemize}
    \item Suppose $\rho$ is of real or complex type. Denote by $d_\rho$ dimension of $\rho$. Let $X_\rho \in \mathbb{C}^{n d_\rho \times d_\rho}$ be a vector of representations $\rho(u_g)$ for all $g \in G$. We observe $n d_\rho \times n d_\rho$ matrix 
    \begin{equation}\label{eq:model}
    Y_\rho = \frac{ \lambda_\rho}{n} X_\rho X_\rho^* + \frac{1}{\sqrt{n d_\rho}} W_\rho,
    \end{equation}
    where the noise matrix $W_\rho$ is a GOE or GUE respectively. 
    \item Suppose $\rho$ is of quaternionic type. Let $d_\rho$ be dimension of $\rho$ over quaternionic field $\bb H$ (that equals half dimension over complex field). In this case, $X_\rho \in \mathbb{C}^{2n  d_\rho \times 2 d_\rho}$ and we observe $2n d_\rho \times 2n d_\rho$ matrix
    \begin{equation}\label{eq:model_quaternionic}
    Y_\rho = \frac{ \lambda_\rho}{n} X_\rho X_\rho^* + \frac{1}{\sqrt{n d_\rho}} W_\rho,
    \end{equation}
    where $W_\rho$ is GSE matrix. 
\end{itemize}
In all these cases, scalar $\lambda_\rho \in \mathbb{R}, \lambda_\rho\ge 0$ denotes the signal-to-noise ratio.
\end{definition}

We can now give a formal statement of the results for general finite groups. 
As a reminder, we show that the computational threshold in the low-degree sense remains the same as in single-frequency case despite having additional information about the signal. %

\begin{theorem}\label{thm:main_all_priors}%
Let $G$ be a finite group of size $L$ with $L > 2$ that does not depend on $n$.

Consider Gaussian synchronization model as in \Cref{def:GSM} and denote by $\lambda_{\text{max}}$ the maximum signal-to-noise ratio value among all the frequencies, $\lambda_{\text{max}} = \max_{\rho \in \Psi} \lambda_\rho$.

If $\lambda_{\text{max}} \le 1$, then for all $D = o(n^{1/3})$, it holds that $\| L^{\le D}_n\| = O(1)$.
\end{theorem}
\begin{remark}
    The case $L = 2$ can be treated using the same proof technique with slight technical adjustments. We omit the proof for this case here because it corresponds to $\bb Z_2$-synchronization, where the statistical and computational thresholds are known to coincide with the spectral one \cite{Deshpande}.
\end{remark}

Here the phase transition is sharp due to BBP-transition as in the angular synchronization. Therefore, we will proceed with the assumption that all $\lambda_\rho$ values are identical as otherwise we can set all values to the largest one. The sharpness also implies that the same threshold applies to a model with only a partial list of frequencies.

\subsection{Alternative view on the multi-frequency model}\label{sec:noisy_indicator}

Denote by $g^* = (g_1^\ast, \dots, g_n^\ast)$ the true signal. 

We assume that for each pair $(k, j)$ we obtain a noisy indicator $z_{kj} : G \to \C $ of the true value of $g_k^\ast (g_j^\ast)^{-1}$, i.e., 
$$
z_{kj}(g) = \gamma  \bm 1\{g=g_k^\ast (g_j^\ast)^{-1}\} + w_{kj}(g),
$$
where $w_{kj}(g)\sim\C\mathcal{N}(0,1)$ are independent for $k, j, $ and $g \in G$. Here we assume the complex standard Gaussian variables, however, it is equivalent to the model with the real Gaussian variables with the scaled signal-to-noise ratio $\gamma_{\text{real}} = \gamma \sqrt{2}$. 

We will show that this model is equivalent to the multi-frequency model over all irreducible representations~\eqref{eq:def_mfreq_groups_intro} with $\gamma =  \lambda \sqrt{\frac{L}{n}}$.

For each pair $(k, j)$ consider an $L\times L $ matrix whose rows and columns are indexed by group elements 
$$
\tilde{Y}_{kj}(t, s) = z_{kj}(t s^{-1}) =  \gamma \bm 1\{g s^{-1}=g_k^\ast (g_j^\ast)^{-1}\} + w_{kj}(t s^{-1}).
$$

Consider $\rho_{\text{reg}}$ the left regular representation of $G$ defined by $\rho_{\text{reg}}(t) e_s = e_{ts}$ for $t, s \in G$ and extending linearly. In particular, the definition implies that matrix elements of $\rho_{\text{reg}}$ indexed by group elements are
$$
[\rho_{\text{reg}}(h)]_{t, s} = \begin{cases}
    0,\quad ts^{-1} \ne h \\
    1,\quad ts^{-1} = h 
\end{cases} = \bm 1(ts^{-1} = h).
$$
Therefore, we can express matrix $Y_{kj}$ as
$$
\tilde{Y}_{kj}= \gamma \rho_{\text{reg}}(g_k^\ast (g_j^\ast)^{-1}) + \sum_g w_{kj} (g) \rho_{\text{reg}} (g).
$$

We can rewrite the noise term as the Gaussian Cayley matrix which is given by 
$$
W^{\text{Cayley}} \colonequals \sum_{g} w(g) \rho_{\text{reg}}(g),
$$
where $w(g)$ are independent Gaussian random variables. Combining everything together implies
$$
\tilde{Y}_{kj}= \gamma \rho_{\text{reg}}(g_k^\ast (g_j^\ast)^{-1}) + W^{\text{Cayley}}_{kj},
$$
where $W^{\text{Cayley}}_{kj}$ are independent Gaussian Cayley matrices for $k, j = 1, \dots, n$. 

By Peter-Weyl theorem, $\rho_{\text{reg}}$ decomposes into the direct sum of the irregular representations, which we denote here by $\rho_1, \dots, \rho_K$. Hence, there exists a deterministic unitary matrix $U\in\C^{L\times L}$ such that 
$$
U \tilde{Y}_{kj} U^\ast = \gamma \begin{pmatrix}
    \rho_1(g_k^\ast (g_j^\ast)^{-1}) & & \\
    & \ddots & \\
    & & \rho_K(g_k^\ast (g_j^\ast)^{-1}) 
\end{pmatrix} + \sum_g w_{kj}(g) \begin{pmatrix}
    \rho_1(g) & & \\
    & \ddots & \\
    & & \rho_K(g) 
\end{pmatrix} .
$$
By \cite[Lemma7]{bandeira2022concentration}, the last term has the same distribution as the random block-diagonalized matrix with blocks 
$$
    \sqrt{\frac{L}{d_{\rho_1}}}W_{kj}^{\rho_1}, \dots ,  \sqrt{\frac{L}{d_{\rho_K}}}W_{kj}^{\rho_K},
$$
where $d_{\rho}$ is the dimension of representation $\rho$ and $W_{kj}^\rho$ is a random matrix with independent Gaussian entries. The matrices are independent for each $k, j$, and $\rho$.

The off-diagonal blocks do not carry information on the signal and we can discard them. We arrive at the canonical model where $(k, j)$ observation is $d_{\rho}\times d_{\rho}$ matrix 
$$
Y^{\rho}_{kj} = \frac{\lambda}{n} \rho(g_k^\ast) \rho\lt((g_j^\ast)^{-1}\rt) + \frac{1}{\sqrt{nd_{\rho}}} W_{ij}^{\rho},
$$
for $\gamma = \lambda \sqrt{\frac{L}{n}}$. This model coincides with the one defined in \Cref{def:GSM}.

\section{Synchronization over finite groups: proofs}

\paragraph{Proof outline} The first step of the proof is to rewrite the LDLR in terms of integer random variables $n_g$ that correspond to counts of realizations of a group element $g$ in a signal vector $u \in G^n$. The final expression is given in \Cref{lm:ldlr_counts_ng} and we devote \Cref{ssec:ldlr_counts_ng} to its derivation. %

Subsequently, we recursively eliminate one variable at a time by computing conditional expectation. The recursion step is given in \Cref{lm:expectation_T_step_k_main} with its more precise variant provided in \Cref{lm:T_recurrent}. This is the most technical step and we defer the proof to \Cref{asec:proof_finite}. We outline the main steps in \Cref{sec:combinatorics_proof_finite}, and we demonstrate the detailed proof for the special case of a cyclic group of size 3 in \Cref{sec:finite_3_group}.  

\subsection{Multinomial expression of $\|L_n^{\le D}\|^2$ for general $G$}\label{ssec:ldlr_counts_ng}

The main result of this section is \Cref{lm:ldlr_counts_ng} which provides $\|L_n^{\le D}\|^2$ expression in terms of integer counts which we define below. The final expression does not contain any particularities of the representation types of the model, however, in the intermediate steps we treat each representation type slightly different. While the proof for the real and complex case can be easily followed side by side, the quaternionic case contains minor technical differences that hurdles readability. 
The ideas for the main steps for all three types can be grasped from the proof provided in the main part, and for the details on the quaternionic case we refer the curious reader to~\Cref{asec:quaternionic}. 

\begin{proposition}\label{prop:ldlr_basic_expression}
Let $G$ be a finite group and denote by $\pi$ the signal prior distribution, i.e., the normalized Haar measure over $G$. Let $X, X^\prime \sim \pi$. 

Suppose that all frequencies $\rho$ are either of real or complex type. Set $\beta_\rho =1$ if $\rho$ is of real type and $\beta_\rho = 2$ if $\rho$ is of complex type. %

We have
\begin{equation*}%
\| L^{\le D}_n \|^2 = \sum_{d=0}^D \frac{1}{d!}  \E_{X, X^\prime} \Big ( \frac{\lambda^2}{n} \sum_{\rho} \frac{\beta_\rho d_\rho}{2} \| X_{\rho}^* X^\prime_{\rho} \|^2_F \Big)^d .
\end{equation*}
Equivalently, we can write
\begin{equation}\label{eq:ldlr_main}
    \| L^{\le D}_n \|^2 = \sum_{d=0}^D \frac{1}{d!}  \E_{X} \Big ( \frac{\lambda^{2}}{n} \sum_{\rho} \frac{ \beta_\rho d_\rho}{2} \| X_{\rho} I_{\rho, n} \|^2_F \Big)^d,
\end{equation}
where $I_{\rho, n} \in \bb C^{n d_\rho \times d_\rho}$ is a matrix constructed by stacking vertically $n$ identity matrices of dimension $d_\rho$. 
\end{proposition}

To prove this proposition we will again apply \Cref{lm:ldlr_gaussian_additive_model} on LDLR for the Gaussian additive model defined over $\R $ or $\C $ (for quaternionic case, see \Cref{lm:ldlr_gaussian_additive_model_quaternionic} and \Cref{prop:ldlr_basic_expression_quaternionic}).
\begin{proof}%
Similarly to the proof for angular synchronization, we rewrite the model using asymmetric noise matrix and adjust the variance as following:
$$
\tilde{Y}_\rho = \frac{\lambda}{n} X_\rho X_\rho^* + \sqrt{\frac{2}{n d_\rho}} \tilde{W}_\rho.
$$
Here $\tilde{W}_\rho$ is a random matrix whose entries are independent Gaussian random variables of the respective type. %

In the notation of \Cref{lm:ldlr_gaussian_additive_model}, the signal $\theta$ consists of concatenated flattened matrices 
$$
\frac{\lambda d_\rho}{\sqrt{2 n}} X_\rho X_\rho^*  \in \C^{n d_\rho \times n d_\rho}.
$$
We denote each flattened vector corresponding to representation $\rho$ by $\theta_\rho$. 
We stack signal components over different representations which may be of different type. We incorporate this construction the same way as in the proof of~\Cref{lm:ldlr_angular_cyclic_base_expression}, i.e., by decomposing the observation into real-valued observations over separate components. It is easy to see that we obtain a sum over different types of representations with the corresponding coefficient $\beta$ in place of the inner product. 

It now remains to compute the expression inside $d$-th exponent. Denote by $X_\rho$ and $X_\rho^\prime$ the signal copies corresponding to $\theta_\rho$ and $\theta_\rho^\prime$. We have
\begin{align*}
\Re \sum_\rho \beta_\rho \langle \theta_\rho, \theta_\rho^\prime  \rangle &= \Re \lt( \sum_\rho  \frac{\lambda^2 d_\rho \beta_\rho }{2 n }\langle X_\rho X_\rho^*, X_\rho^\prime (X_\rho^\prime)^* \rangle_F \rt) \\
&= \Re \lt( \sum_\rho  \frac{\lambda^2 d_\rho \beta_\rho }{2 n }\tr ((X_\rho^\prime)^* X_\rho X_\rho^* X_\rho^\prime )  \rt) \\
&= \sum_\rho \frac{\lambda^2 d_\rho \beta_\rho  }{2 n }\|X_\rho^* X_\rho^\prime \|^2_F.
\end{align*}

Finally,  we replace $\|X_\rho^* X_\rho^\prime \|^2_F$ by $\| X_{\rho} I_{\rho, n} \|^2_F$ due to its invariance under left and right translation by a group element. We thus rewrite the LDLR as
\begin{align*}
\| L^{\le D}_n \|^2 &= \sum_{d=0}^D \frac{1}{d!}  \E_{X} \Big ( \frac{\lambda^{2}}{n} \sum_{\rho} \frac{ \beta_\rho d_\rho}{2} \| X_{\rho} I_{\rho, n} \|^2_F \Big)^d,
\end{align*}
giving the result.
\end{proof}

We now further rewrite the ratio using integer random variables which count how many times every element $g \in G$ was sampled in the signal vector $X$. More specifically, we will use the counts $n_g = |\{j: X_j = g  \}|\in \bb N$, where $g$ is an element of $G$. Note that since the group is finite, we have $L$ such random variables. Moreover, similarly to the cyclic group, these variables are distributed according to the multinomial distribution with probabilities $1/L$ and number of trials $n$. The following lemma is a generalized version of \Cref{lm:ldlr_counts_ng_cyclic} for general finite group. %
\begin{lemma}\label{lm:ldlr_counts_ng} Suppose that $G$ is finite and let $n_g \colonequals n_g(X) = |\{j: X_j = g  \}|\in \bb N$ for every $g \in G$. Denote $L \colonequals |G|$. The second moment of the LDLR is
\begin{align}\label{eq:ldlr_counts_ng}
\| L^{\le D}_n \|^2 &= \sum_{d=0}^D \frac{1}{d!}  \frac{\lambda^{2d}}{n^d} \E_{X} \Big (  \frac{L-1}{2}\sum_{g \in G} n_g^2 - \frac{1}{2} \sum_{\substack{g, f \in G\\ g\ne f}}n_g n_f \Big)^d .
\end{align}
\end{lemma}
\begin{proof}
In this proof, we assume frequencies are either real or complex; however, the formulation also applies to cases with quaternionic representations. For the latter, see the proof in \Cref{asec:quaternionic}.

Recall the notation $\rho(g)\in \C^{d_\rho \times d_\rho}$ for a complex-valued representation of a group element $g\in G$. Using this notation and expression \eqref{eq:ldlr_main}, we can write the LDLR as 
   \begin{align}
\| L^{\le D}_n \|^2 &= \sum_{d=0}^D \frac{1}{d!}  \E_{X} \Big ( \frac{\lambda^{2}}{n} \sum_{\rho} \frac{ \beta_\rho d_\rho}{2}  \Bigl\| \sum_{g \in G} n_g \rho(g)\Bigr\|^2_F \Big)^d \nonumber \\
&= \sum_{d=0}^D \frac{1}{d!}  \E_{X} \Big ( \frac{\lambda^{2}}{n} \sum_{\rho} \frac{ \beta_\rho }{2}  \sum_{i,j=1}^{d_\rho} \Bigl|\sum_{g \in G}n_g \sqrt{ d_\rho}\rho(g)_{ij}\Bigr|^2 \Big)^d.\label{eq:ldlr_PW_basis_expansion}
\end{align}

By the Peter-Weyl theorem on the orthogonality of matrix coefficients, the set of basis functions $ \{g \mapsto\sqrt{d_{\rho}} \rho(g)_{ij}\}_{\rho, i, j}$ forms an orthogonal basis with respect to Hermitian inner product on $L^2(G)$. Recall that for a finite group of size $L$, the inner product between $f(g) $ and $h(g)$ can be expressed as 
$$
\ip{f, g}  = \frac{1}{L}\sum_g f(g) \bar h(g).
$$%
Thus, for fixed $\rho, i, j$, we can view $\Bigl|\frac{1}{L}\sum_{g \in G}n_g \sqrt{d_\rho}\rho(g)_{ij}\Bigr|^2$ as the squared inner product between vector with elements $n_g$ and the basis function $ \{g \mapsto\sqrt{d_{\rho}} \rho(g)_{ij}\}_{\rho, i, j}$ in $L^2(G)$ sense. Due to basis invariance of $\ell_2$-norm, we have
$$
\frac{1}{L^2} \sum_{\rho \in \Psi_{\text{all}}} \sum_{i, j = 1}^{d_\rho} \lt |\sum_{g \in G}n_g \sqrt{d_\rho}\rho(g)_{ij}  \rt |^2 = \|(n_g)_{g\in G} \|^2_{\ell_2} = \frac{1}{L} \sum_{g \in G} n_g^2, 
$$
where $ \Psi_{\text{all}}$ denotes the complete list of irreducible representations including trivial representation. 

Observe that the summation in \eqref{eq:ldlr_PW_basis_expansion} is taken over $\Psi$ which excludes the trivial representation and considers only one complex representation from each conjugate pair. The latter allows us to remove coefficient $\beta_\rho = 2$ for complex representations. By adding and substracting also a trivial representation, we complete the list to the full list~$\Psi_{\text{all}}$.

Combining everything together, we have
\begin{align*}
\sum_{\rho} \frac{\beta_\rho d_\rho}{2}  \sum_{ij} \Bigl|\sum_{g \in G}n_g\rho(g)_{ij}\Bigr|^2 &= \frac{L^2}{2} \lt(\frac{1}{L} \sum_{g\in G} n_g^2 - \lt(\frac{1}{L} \sum_{g \in G} n_g\rt)^2 \rt)\\
&= \frac{L-1}{2}\sum_{g \in G} n_g^2 - \frac{1}{2} \sum_{\substack{g, f \in G\\ g\ne f}}n_g n_f.
\end{align*}
In the first equality we substracted a term corresponding to the trivial representation. Substituting the above expression back to \eqref{eq:ldlr_PW_basis_expansion}, we obtain the desired expression. 
\end{proof}
Since the group is finite, we can enumerate all group elements in arbitrary order and refer to counts $n_g$ with an integer subscript $n_\ell$, $\ell = 0, \dots, L-1$.  Since the signal prior is the normalized Haar measure over $G$, the counts $n_0, \dots, n_{L-1}$ follow multinomial distribution with parameters $n$ and probabilities $1/ L$. Therefore, we can equivalently write \eqref{eq:ldlr_counts_ng} as 
\begin{equation}\label{eq:ldlr_counts_nL}
    \| L^{\le D}_n \|^2 = \sum_{d=0}^D \frac{1}{d!}  \frac{\lambda^{2d}}{n^d} \E \Big (  \frac{L-1}{2}\sum_{\ell = 0}^{L-1} n_\ell^2 - \frac{1}{2} \sum_{\substack{\ell, k = 0\\ \ell\ne k}}^{L-1} n_\ell n_k \Big)^d ,
\end{equation}
where the expectation is taken with respect to randomness of $n_\ell$.

\subsection{Combinatorial analysis}\label{sec:combinatorics_proof_finite}

In this section, we outline the main steps of the argument. This section is structured as follows.

We first decompose the quantity inside the expectation into a product of certain random variables that depend only on a subset of random counts $n_1, \dots, n_k$. This decomposition is given in \Cref{lm:S_l_finite_group}. To compute the expectation, we iteratively apply the tower property with respect to only one variable. We provide the recursion step in \Cref{lm:expectation_T_step_k_main}. After simplifying the expression, we arrive at the final bound on LDLR moment which depends only on $\lambda, L$, and $d$ in \Cref{lm:ldlr_final_finite_groups}. Together with \Cref{prop:polylogarithm-function}, this allows us to finish the proof of \Cref{thm:main_all_priors}.

We defer the detailed proofs to \Cref{asec:proof_finite} due to their technicality. For illustration of the proof technique, we provide the proof for the special case, $L=3$, in \Cref{sec:finite_3_group}. 
Recall that we define $S_L$ as
$$
S_L = \frac{L-1}{2}\sum_{\ell=0}^{L-1} n_\ell^2 - \sum_{\substack{j, k = 0 \\ j \ne k}}^{L-1} n_j n_k. 
 $$
\begin{remark}
    Up to scaling, $S_L$ is $\sum_{\ell = 0}^{L - 1} (n_{\ell} - \frac{1}{L}n)^2$.
    This random variable is well-known in statistics as having the distribution of a \emph{Pearson $\chi^2$ statistic} (sharing a name with the $\chi^2$ distribution that is its limit if $L$ is fixed and $n \to \infty$), and our analysis below may be viewed as a detailed non-asymptotic study of its moments of high order $d = d(n) = \omega(1)$.
    Relatedly, while it is unclear to us whether the constraint $d = o(n^{1/3})$ in our results is optimal, it seems natural that some such threshold of polynomial scaling should appear in our bounds, for example by comparison with the sharp bounds on moments of the binomial distribution (while our case involves the multinomial distribution) of \cite{Ahle-2022-BoundsRawMomentsBinomialPoisson}, which exhibit a natural transition at $d \sim n^{1/2}$.
\end{remark}
 
The idea of the first step is to split the sum into the terms such that first $n - k$ terms do not contain variable $n_k$. This decomposition is provided in the following lemma. 
\begin{lemma}\label{lm:S_l_finite_group} For all $d \ge 0$, we have
    \begin{equation}\label{eq:S_l_finite_group}
    S_L^d = \lt ( \frac{L}{2}\rt)^d \sum_{\substack{d_1, \dots, d_{L-1} \ge 0 \\ d_1 + \dots +d_{L-1} = d}} \binom{d}{d_1 \dots d_{L-1}} \prod_{\ell = 1}^{L-1} \lt ( \frac{L - \ell + 1}{
L - \ell } \rt) ^{d_\ell}  \lt ( \frac{n - \sum\limits_{j=1}^{\ell-1} n_{j} }{L - \ell + 1} - n_{\ell} \rt)^{2d_{\ell}}.
    \end{equation}
\end{lemma}
\begin{proof}%
    We can view this expression as a bilinear form with input vector $(n_0, \dots, n_{L-1})^{\top}$. We can succinctly write
\begin{equation}\label{eq:bilinear}
S_L = (n_0, \dots, n_{L-1})^{\top} \lt(\frac{L}{2} I_{L\times L} - \frac{1}{2} \bm 1_L \bm 1_L^{\top}\rt) (n_0, \dots, n_{L-1}),
\end{equation}
where $I_{L\times L}$ is the identity matrix of size $L$ and $\bm 1_L$ is $L$-dimensional all-ones vector. 

By taking SVD of the bilinear form in \eqref{eq:bilinear} we can obtain the desired decomposition. It can also be verified through straightforward computations of the coefficients associated with quadratic and cross terms. We obtain
\begin{align*}
    S_L &= \frac{L}{2} \bigg ( 2 \Big (\frac{n_0 - n_{L-1}}{2} \Big)^2 + \frac{3}{2} \Big (\frac{n_0 + n_{L-1}}{3} - \frac{2}{3} n_{L - 2} \Big)^2 \\ 
    &\quad+ \dots  + \frac{j}{j - 1} \Big (\frac{n_0 + n_{L-1} + \dots + n_{L - j + 1}}{j} - \frac{j - 1}{j} n_{L-j}\Big)^2 \\ 
    &\quad+ \dots  + \frac{L}{L-1} \Big (\frac{n_0 + n_{L-1} + \dots + n_2}{L} - \frac{L-1}{L}n_1 \Big)^2 \bigg).
\end{align*}

We further rewrite the expression using the fact that $n_0, \dots, n_{L-1}$ sum to $n$, which is deterministic. This allows us to eliminate random variable $n_0$ from the expression. 
\begin{align*}
    S_L &= \frac{L}{2} \bigg ( 2 \Big (\frac{n - n_1 - \dots - n_{L-2}}{2} - n_{L-1} \Big)^2 \\&\quad  + \dots 
    + \frac{j}{j - 1} \Big (\frac{n - n_{1} - \dots - n_{L - j - 1}}{j} - n_{L-j}\Big)^2 \\&\quad + \dots
     + \frac{L}{L-1} \Big (\frac{n}{L} - n_1 \Big)^2\bigg). %
\end{align*}

Taking $d$-th power of this sum and applying binomial theorem yields
\begin{equation}\label{eq:lowdegree_ratio_with_n_ell}
    \begin{split}
          S_L^d = \sum_{d=0}^D \frac{\lambda^{2d}}{n^d d!}  \lt ( \frac{L}{2}\rt)^d \sum_{\substack{d_1, \dots, d_{L-1} \ge 0 \\ d_1 + \dots d_{L-1} = d}} \binom{d}{d_1 \dots d_{L-1}} \prod_{\ell = 1}^{L-1} \lt ( \frac{L - \ell + 1}{
L - \ell } \rt) ^{d_\ell}  \lt ( \frac{n - \sum\limits_{j=1}^{\ell-1} n_{j} }{L - \ell + 1} - n_{\ell} \rt)^{2d_{\ell}},
    \end{split}
\end{equation}
as claimed.
\end{proof}

This transformation has two characteristics that will be useful for computing the expectation. First, we expressed the low-degree ratio as sum of non-negative terms. Second, a variable $n_k$ appears only in the last $k$ terms, and this fact allows us to compute iteratively expectation with respect to only one random variable at a time.

Before proceeding to the next step, we introduce the following notation. First, for $k \in \bb N$, define $T_{k, \alpha}$ for $\alpha = (\alpha_1, \dots, \alpha_k) \in \bb R^k$ as
\[
T_{k, \alpha} \colonequals \prod_{\ell = 1}^{k}  \lt | \frac{n - \sum\limits_{j=1}^{\ell-1} n_{j} }{L - \ell + 1} - n_{\ell} \rt|^{2\alpha_{\ell}}.
\]
Observe, that $T_{k, \alpha}$ is a random variable that depends only on the first $k$ random counts $n_k$. Secondly, define a deterministic constant $C_{k, \alpha}$ as
$$
C_{k, \alpha} \colonequals \prod_{\ell = 1}^{k} \lt ( \frac{L - \ell + 1}{
L - \ell } \rt) ^{d_\ell}.
$$

It is easy to see that with this notation we can rewrite the expectation of $S_L^d$ as
$$
\E S_L^d =  \lt ( \frac{L}{2}\rt)^d \sum_{\substack{d_1, \dots, d_{L-1} \ge 0 \\ d_1 + \dots d_{L-1} = d}} \binom{d}{d_1 \dots d_{L-1}} C_{L-1,(d_1, \dots, d_{L-1})}  \E_{n_1, \dots, n_{L-1}}T_{L-1, (d_1, \dots, d_{L-1})}.
$$
The next lemma provides a recursion on the defined variables $T_{k, \alpha}$. This is a result of taking the conditional expectation with respect to one count $n_k$. We provide here a simplified version of the lemma, the full statement is given in \Cref{lm:expectation_T_step_k}.
\begin{lemma}\label{lm:expectation_T_step_k_main} Assume $L>2$ and let $\alpha \colonequals (\alpha_1, \dots, \alpha_{L-1}) \in \bb R^{L-1}$. Denote the truncated vector $ \alpha_{:k} = (\alpha_1, \dots, \alpha_k)$. Then for the first step we have,
\begin{equation*}
    \begin{split}
        &\E_{\mathbf n_{1:L-1}} T_{L-1, \alpha}
\le  \sum_{\beta=0}^{\ceil{\alpha_{L-1}}} \tilde{C}( \beta) \E_{\mathbf n_{1:L-2}}  \lt ( n - \sum_{j=1}^{L-3} n_j \rt)^{\gamma} T_{L-2, ( \alpha_{:L-3}, \alpha_{L-2} + \beta / 2)},
    \end{split}
\end{equation*}
where $\gamma = \alpha_{L-1}-\beta \ge 0$ and $\tilde{C}( \beta)$ is a deterministic constant depending only on $\alpha_{L-1}$ and $\beta$. 

Let now $\gamma \ge 0$ be any non-negative number, and assume that $\alpha \in \R^k$. For $1 < k \le L-1$, on step $L-k$, the expectation is recursively bounded as
\begin{equation*}
        \begin{split}
            &\E_{\mathbf n_{1:k}}  \lt(n - \sum\limits_{j=1}^{k-1}n_j\rt)^{\gamma}T_{k, \alpha}   \\
            &\quad \le  \sum_{\beta=0}^{\ceil{\alpha_k + \gamma}}\tilde{C}(\beta)\E_{\mathbf n_{1:k-1}} \lt(n - \sum\limits_{j=1}^{k-2}n_j\rt)^{\ceil{\alpha_k + \gamma} - \beta} T_{k-1, (\alpha_{1:k-2},  \alpha_{k-1} + \beta/2)}, 
        \end{split}
    \end{equation*}
where $\tilde{C}(\beta)$ is a deterministic constant that depends on $\beta, \alpha_k$, and $\gamma$.
\end{lemma}

By unrolling the recursion and computing the deterministic coefficients, we arrive at the following bound of LDLR moment. The full proof of the following lemma can be found in \Cref{asec:proof_finite}.
\begin{lemma} \label{lm:ldlr_final_finite_groups} For all $D \le o(n^{1/3})$ and $L = O(1)$, we have
    $$
\| L_n^{\le D}\|^2 \le \sum_{d=0}^D \lambda^{2d} d^{2L} .
$$
\end{lemma}

To finish the proof that $\|L_n^{\le d}\|^2$ is bounded for $\lambda < 1$, we introduce the polylogarithm function and state its convergence region. 
\begin{proposition}\label{prop:polylogarithm-function}
    Define the polylogarithm function (also known as Jonquière's function) of order $s \in \bb C$ and argument $z\in \C$ as the following power series
    $$
    \Li_s(z) = \sum_{k=0}^\infty \frac{z^k}{k^s}.
    $$
    This definition is valid for any $s \in \bb C$ and any $z$ such that $|z| < 1$; in other words, in this regime, the series is convergent. 
\end{proposition}

\begin{proof}[Proof of \Cref{thm:main_all_priors}]

Consequently, by the above lemma, we obtain the following bound
$$
\| L_n^{\le D}\|^2 \le \sum_{d=0}^D \lambda^{2d} d^{2L} \le \sum_{d=0}^\infty \lambda^{2d} d^{2L} = \Li_{-2L}(\lambda^2) \le C.%
$$
where the last inequality holds for any constant $L$ by \Cref{prop:polylogarithm-function}. 
\end{proof}

 \subsection{Special case of combinatorial analysis: $L=3$}\label{sec:finite_3_group}
In this section, we demonstrate the idea of the proof on a special case. It contains all the main steps of the proof but technically simpler than the general case.

\begin{lemma}\label{lm:expectation_finite_group_3}
For $d = o(n^{1/3})$ it holds that
$$
 \E S_3^d \le C n^d d \cdot d! ,
$$
where $C>0$ is an absolute constant.
\end{lemma}
\begin{proof}
By substituting $L=3$ into \eqref{eq:S_l_finite_group}, we get
\begin{align*}
S_3 &= \lt (\frac{3}{2} \rt)^d \sum_{\substack{0 \le d_1, d_2 \le d\\ d_1 + d_2 = d}} \binom{d}{d_1 } C_{2, (d_1, d_2)} \E\limits_{n_1, n_2} T_{2, (d_1, d_2)} \\
&= \lt (\frac{3}{2} \rt)^d \sum_{\substack{0 \le d_1, d_2 \le d\\ d_1 + d_2 = d}} \binom{d}{d_1 } \lt(\frac{3}{2}\rt)^{d_1} 2^{d_2} \lt (\frac{n }{3} - n_1 \rt)^{2d_1} \lt (\frac{n - n_1}{2} - n_2 \rt)^{2d_2}.
\end{align*}

Following the outline in \Cref{sec:combinatorics_proof_finite}, we aim to reduce $T_{2, (d_1, d_2)}$ to $T_{1, (d_1 + k/2)}$ for $k = 0, \dots, d_1$. We accomplish this in two steps: firstly, by marginalizing out one variable, and secondly, by consolidating terms into a unified form.
\paragraph{Conditional expectation w.r.t. $n_2$}
To simplify the following explanation, we will rewrite the last factor as 
$$
\lt (\frac{n - n_1}{2} - n_2 \rt)^{2d_2} = \frac{1}{2^{d_2}}\lt (n_0 - n_2 \rt)^{2d_2}.
$$
This step is optional, and a similar argument can be used without this transformation.

Let $u \in G^n$ be a random vector of size $L$ supported over $G$ and be sampled from the normalized Haar uniform measure. Denote elements of $G$ by $\omega_0, \omega_1, \omega_2$. Consider the random variable $\xi_1 \colonequals \mathds{1} (u_1 = \omega_0) - \mathds{1} (u_1 = \omega_2)$ conditioned on $\mathds{1} (u_1 = \omega_1)$. 
When $u_1 \ne \omega_1$, $u_1$ can take only value $\omega_0$ or $\omega_2$, and therefore $\xi_1$ is distributed as a Rademacher random variable (i.e. takes value $\pm 1$ with probability $1/2$), and always zero otherwise. Due to independence of coordinates $u_j$, we have that $n_0 - n_2 | n_1 $ is distributed as sum of $(n-n_1)$ Rademacher random variables. %
Using this observation and the tower property of the expectation we can write 
\begin{equation*}
    \begin{split}
        \E S_3^d &= \lt(\frac{3}{2} \rt) ^{d}  \E_{n_0, n_1, n_2} \sum_{\substack{0 \le d_1, d_2 \le d, \\ d_1 + d_2 = d}}  \binom{d}{d_1 \  d_2}  \frac{3^{d_1}}{2^{d_2}} \Big (\frac{n}{3} - n_1\Big)^{2d_1} (n_0 - n_2)^{2d_2}  \\
        &=  \lt(\frac{3}{2} \rt) ^{d} \sum_{\substack{0 \le d_1, d_2 \le d, \\ d_1 + d_2 = d}} \binom{d}{d_1 }  \frac{3^{d_1}}{2^{d_2}}  \E_{n_1} \bigg [ \Big(\frac{n}{3} - n_1\Big)^{2d_1} \E_{n_0, n_2} \Big [ (n_2 - n_0)^{2d_2} \big  | n_1 \Big ] \bigg] \\
        &\quad \le  \lt(\frac{3}{2} \rt) ^{d} \sum_{\substack{0 \le d_1, d_2 \le d, \\ d_1 + d_2 = d}} \binom{d}{d_1 }  \frac{3^{d_1}}{2^{d_2}}  \E_{n_1} \Big [ \Big(\frac{n}{3} - n_1\Big)^{2d_1} 2^{d_2} (n - n_1)^{d_2} d_2! \Big]. %
    \end{split}
\end{equation*}
\paragraph{Unifying factors containing $n_1$}
We note further that $n_1 - \frac{n}{3} = n_1 - \E n_1$ is a sum of centered iid variables $\sum\limits_{i=1}^n[\mathds{1}(u_i = \omega_1) - \frac{1}{3}]$ with variance~$ \frac{2}{9}$. To transform both terms in the product to the same form, we use the binomial theorem on term $(n-n_1)^{d_2}$,
$$
( n - n_1)^{d_2} = \lt ( \lt (\frac{n}{3 } - n_1\rt)+ \frac{2n}{3} \rt)^{d_2} = \sum_{k=0}^{d_2} \binom{d_2}{k}\lt (\frac{n}{3 } - n_1\rt)^k  \lt (\frac{2n}{3}\rt)^{d_2 - k}.
$$ 
We can proceed with the proof right away, but to highlight the connection to the proof in the general case, observe that the above display implies 
$$
\E T_{2, (d_1, d_2)} \le \sum_{k=0}^{d_2} 2^{d_2} \binom{d_2}{k} \lt(\frac{2}{3}\rt)^{\gamma} T_{1, (d_1 + k / 2)} n^{\gamma},
$$
where $\gamma = d_2 - k$. This expression precisely matches that given by \Cref{lm:expectation_T_step_k_main} with $\Tilde{C}(k) = 2^{d_2} \binom{d_2}{k} \lt(\frac{2}{3}\rt)^{\gamma}$.

Given that $n_1 \le n$ is a bounded random variable with an existing moment-generating function, we can utilize \Cref{prop:clt_moments} to establish a bound on the $\alpha$-th moment, where $\alpha = d_1 + k / 2$.
\begin{equation*}
    \begin{split}
        \E S_3^d &\le   \lt(\frac{3}{2} \rt) ^{d} \sum_{\substack{0 \le d_1, d_2 \le d, \\ d_1 + d_2 = d}} \frac{d!}{d_1!} 3^{d_1} \E_{n_1} \Big [  \sum_{k=0}^{d_2}  \binom{d_2}{k} \Big|\frac{n}{3} - n_1\Big|^{2d_1+k} \Big( \frac{2n}{3}\Big)^{d_2-k}  \Big] \\
        & \le 4\cdot \lt(\frac{3}{2} \rt) ^{d} \sum_{\substack{0 \le d_1, d_2 \le d, \\ d_1 + d_2 = d}} \frac{d!}{d_1!}   3^{d_1}   \sum_{k=0}^{d_1}  \binom{d_2}{k} \frac{2^{ d_1 +  d_2} n^{d_1 + d_2 - k / 2}}{3^{2d_1 +d_2}} \Gamma \lt( d_1 + \frac{k}{2} \rt) \\
        & = 4\cdot n^d  \sum_{\substack{0 \le d_1, d_2 \le d, \\ d_1 + d_2 = d}} \frac{d!}{d_1!}   \sum_{k=0}^{d_1} n^{- k / 2} \frac{d_2!}{k! (d_2 - k)!} \Gamma \lt ( d_1 + \frac{k}{2} \rt).
    \end{split}
\end{equation*}

The goal now is to simplify the term inside the inner sum. We will leverage term $n^{-k/2}$ to control other terms depending on $d, d_1, k$. We bound each $k$-th term in the sum by bounding the binomial coefficient with $d_2^k$ and the gamma function with $d_1! d^{k/2}$,
\begin{equation*}
        n^{- k / 2} \frac{d_2!}{k! (d_2 - k)!} \Gamma \lt( d_1 + \frac{k}{2} \rt) \le  \frac{d_2^k \cdot d_1!  d^{k/2}}{n^{k/2}}   \le d_1! \left(\frac{d^3 }{n}\right)^{k / 2}.%
\end{equation*}

Substituting this bound into the sum and using the fact that $d = o(n^{1/3})$ we obtain the desired inequality
\begin{equation*}
        \E S^d = 4\cdot  n^d d! \sum_{d_1=0}^d \sum_{k=0}^{d_1} \left(\frac{d^3 }{n}\right)^{k / 2}  \le 8 n^d d^2 \cdot d!.
\end{equation*}

\end{proof}

\begin{proof}[Proof of \Cref{thm:main_all_priors} for $L=3$]
    
We substitute the result of \Cref{lm:expectation_finite_group_3} to LDLR moment. We get the polylogarithm function of order $-2$ and argument $\lambda$. Hence, by \Cref{prop:polylogarithm-function}, we obtain that for $ \lambda < 1$
$$
\| L^{\le D}_n \|^2 \le 8 \sum_{d=0}^D d^2 \lambda^{2d} \le8 \sum_{d=0}^\infty d^2 \lambda^{2d} \le  C,%
$$
where $C>0$ is an absolute finite constant (that depends only on $\lambda$). 
This completes the proof of \Cref{thm:main_all_priors} for finite groups of size~$3$. 

\end{proof}

\section{Future directions}

This study establishes the computational threshold in the Gaussian synchronization model with multiple frequencies over the finite groups and $SO(2)$. Despite this progress, the landscape of statistical and computational phase transitions is not yet fully understood. We outline several potential directions for further exploration. 

\paragraph{Statistical thresholds}
In a general synchronization model with multiple frequencies, the precise value of statistical threshold is currently unknown. The existing upper bound for the threshold for synchronization over finite groups surpasses the spectral one only for $L \ge 11$, which leaves the question of the presence of the statistical-to-computational gap for values $3 \le L < 11$ open. Information-theoretical thresholds can be analyzed through studying the landscape of the replica potential which has been derived in \cite{yang2024asymptotic}. In their work, the authors provide the precise value of the statistical threshold for synchronization over $SO(2)$ with a single frequency. 

\paragraph{Synchronization over infinite groups}
Numerous applications, including Cryo-EM, involve synchronization over infinite compact groups like $SO(3)$ and, more broadly, $SO(d)$ for $d \ge 3$. These groups are not finite and require a different approach compared to the one presented in the current study.

\paragraph{Non-constant number of frequencies} Numerical simulations in \cite{Gao2019MultiFrequencyPS} suggest the possibility of surpassing the spectral threshold using an efficient algorithm when the number of frequencies diverges with the dimension of the signal. %
Assuming the low-degree conjecture, our results suggest failure of polynomial-time algorithms for constant number of frequencies $L$; however, the computational transition for diverging number of frequencies opens up another interesting direction for future research. 

\newpage
\appendix

\newpage
\printbibliography

\newpage
\appendix

\section{Quaternionic representations}\label{asec:quaternionic}

\begin{lemma}[LDLR of the Gaussian additive model, quaternionic case]
\label{lm:ldlr_gaussian_additive_model_quaternionic}
Let $\theta$ be a $2 N\times 2$-dimensional vector defined over $\bb H$ and expressed over $\C$. Suppose it is drawn from some prior distribution $\mathcal P_N$. Let $Z$ be a random vector of dimension $2 N\times 2$ with independent Gaussian entries of the quaternionic type (as in \Cref{def:gaussian_type}). We define $\bb P_n$ and $\bb Q_n$ as follows.
\begin{itemize}
    \item Under $\bb P_N$, observe $Y = \theta + Z$ (planted distribution).
    \item Under $\bb Q_N$, observe $Y = Z$ (null distribution)
\end{itemize}
    Set $\beta = 2$. Then the low-degree likelihood ratio between $\bb P_N$ and $\bb Q_N$ is 
    $$
    \|L_n^{\le D}(Y)\|^2 = \E_{\theta_1, \theta_2 \sim P_N} \lt [\sum_{d=0}^D \frac{1}{d!} \beta_\rho  \Re\lt( \langle \theta_1, \theta_2  \rangle_F \rt)^d \rt],
    $$
    where $\theta_1$ and $\theta_2$ are drawn independently from $\mathcal{P}_N$.
\end{lemma}
\begin{proof}
    The idea of the proof for the quaternionic observation closely follows the one for the complex case, namely, we separate the signal into components. We can easily extract the coefficients corresponding to the basis vectors $1, i, j, k $ from matrix representation by taking real or imaginary part of the matrix entries in one block. 

Similarly to the complex case, we adjust the variance by multiplying each component by $2$. Recall that the real part of the inner product of two scalar quaternionic variables $x_1 =  a_1+ b_1 i + c_1 j + d_1 k$ and $x_2 =  a_2+ b_2 i + c_2 j + d_2 k$ is
\begin{align*}
\Re \lt (\langle x_1, x_2 \rangle \rt)  %
&= a_1 a_2 + b_1 b_2 + c_1 c_2 + d_1 d_2 = \lt \langle \begin{pmatrix}
    a_1 \\b_1 \\c_1 \\ d_1
\end{pmatrix}, \begin{pmatrix}
    a_2 \\b_2 \\c_2 \\ d_2
\end{pmatrix} \rt \rangle.
\end{align*}

The expression for the real part of the inner product remains the same if we represent $x_1, x_2$ over complex numbers (as $2\times 2$ blocks) and take half the Frobenius inner product. Hence, we can follow the same steps as in the complex case and obtain 

$$
\|L_n^{\le D}(Y)\|^2 = \E_{\theta, \theta^\prime \sim P_N} \lt [\sum_{d=0}^D \frac{1}{d!}   \lt(4 \Re \frac 1 2 \langle \theta_1, \theta_2  \rangle_F \rt)^d \rt] =  \E_{\theta, \theta^\prime \sim P_N} \lt [\sum_{d=0}^D \frac{1}{d!}   \lt(2 \Re \langle \theta_1, \theta_2  \rangle_F \rt)^d \rt].
$$%

\end{proof}

\begin{proposition}\label{prop:ldlr_basic_expression_quaternionic}
Let $G$ be a finite group and denote by $\pi$ the signal prior distribution, i.e., the normalized Haar measure over $G$. Let $X, X^\prime \sim \pi$. 

Suppose that all frequencies $\rho$ are either of real or complex type.  Let $\beta_\rho =1$ if $\rho$ is of real type and $\beta_\rho = 2$ if $\rho$ is of complex or quaternionic type.  %

We have
\begin{equation*}%
\| L^{\le D}_n \|^2 = \sum_{d=0}^D \frac{1}{d!}  \E_{X, X^\prime} \Big ( \frac{\lambda^2}{n} \sum_{\rho} \frac{\beta_\rho d_\rho}{2} \| X_{\rho}^* X^\prime_{\rho} \|^2_F \Big)^d .
\end{equation*}
Equivalently, we can write
\begin{equation}\label{eq:ldlr_main_quaternionic}
    \| L^{\le D}_n \|^2 = \sum_{d=0}^D \frac{1}{d!}  \E_{X} \Big ( \frac{\lambda^{2}}{n} \sum_{\rho} \frac{ \beta_\rho d_\rho}{2} \| X_{\rho} I_{\rho, n} \|^2_F \Big)^d,
\end{equation}
where $I_{\rho, n} \in \bb C^{n d_\rho \times d_\rho}$ is a matrix constructed by stacking vertically $n$ identity matrices of dimension $d_\rho$ if $\rho$ is of real or complex type. For quaternionic type, $I_{\rho, n} \in \bb C^{2 n d_\rho \times 2 d_\rho}$ consists of $n$ stacked vertically identity matrices of size $2d_\rho \times 2d_\rho$. 
\end{proposition}
\begin{proof}
The proof follows exactly the steps of the proof of \Cref{prop:ldlr_basic_expression}, except for minor technical differences.

First the signal $\theta_\rho$ is a matrix of size $2 d_\rho \times 2$ with complex entries rather than a vector. Nevertheless, in the same manner as in proof of \Cref{lm:ldlr_gaussian_additive_model_quaternionic} and \Cref{lm:ldlr_gaussian_additive_model}, we can decompose it to the components and see that we can sum over different types of representation in place of the inner product. The expression inside $d$-th exponent is computed in the same way, in this case using the expression from \Cref{lm:ldlr_gaussian_additive_model_quaternionic}.
\end{proof}

\begin{proof}[Proof of \Cref{lm:ldlr_counts_ng} for quaternionic case]
    It remains only to argue that coefficient $\beta_\rho$ vanishes in quaternionic case as well. The reason lies in the fact that in Peter-Weyl theorem on the orthogonality of matrix coefficients, the functions are taken over complex numbers. In particular, when the list contains quaternionic representation, the respective orthogonal basis is $\{ g \mapsto \sqrt{d_\rho^\C \rho(g)_{ij}}\}$, where $d_\rho^\C = d_\rho$ for real and complex-type representations and $d_\rho^\C = 2 d_\rho$ for the quaternionic type. Since $\beta_\rho = 2$, it cancels out with constant $2$ appearing from the discrepancy between dimension over quaternionic and complex field. By repeating all the remaining steps of the proof, we arrive at the unified expression for all representation types. 
\end{proof}

\section{Auxiliary lemmas}

\begin{lemma}\label{prop:clt_moments} Let $X_1, \dots, X_n \in \mathbb{R}$ be iid centered random variables such that their moment generating function exists. Then for $\alpha \ge 0, \alpha = o(n)$ the moments are bounded as following
$$
\E \left |\sum_{i=1}^n X_i \right|^{2\alpha} \le 4 \cdot 2^{\alpha} \Gamma(2\alpha+1) \sigma^{2\alpha}  n^{\alpha} ,
$$
where $\sigma^2 := \mathbb{E} | X_1 |^2$. 
\end{lemma}
\begin{proof}

Denote for convenience, $S := \sum_{i=1}^n X_i$ and $M(\lambda) = \E \exp(\lambda X_1) $. Then MGF of $S$ is $M_n(\lambda) = M(\lambda)^n$. By integral identity,
\begin{equation*}
\begin{split}
    \E |S|^{2\alpha} &\le \int_{0}^\infty [\mathbb{P} (S \ge s^{1/(2\alpha)}) + \mathbb{P} (S \le -s^{1/(2\alpha)})]ds \\
    &\le \inf_{\lambda \ge 0}  \int_{0}^\infty  e^{-\lambda t} M_n(\lambda) 2\alpha t^{2\alpha -1} dt + \inf_{\lambda \ge 0}  \int_{0}^\infty e^{-\lambda t} M_n(-\lambda) 2\alpha t^{2\alpha -1} dt \\
    &=  \inf_{\lambda \ge 0}\frac{M_n(\lambda)}{\lambda^{2\alpha}} 2\alpha \cdot \Gamma(2\alpha) +  \inf_{\lambda \ge 0}\frac{M_n(-\lambda)}{\lambda^{2\alpha}} 2\alpha \cdot \Gamma(2\alpha) \\
    &= \Gamma(2\alpha + 1) \left(  \inf_{\lambda \ge 0}\frac{M(\lambda)^n}{\lambda^{2\alpha}} +  \inf_{\lambda \ge 0}\frac{M(-\lambda)^n}{\lambda^{2\alpha}} \right).
\end{split}
\end{equation*}

For the first term, the first-order optimality condition for function $\log(\lambda^{-2\alpha} M(\lambda)^n)$ is 
$$
\frac{\lambda M'(\lambda)}{M(\lambda)} =\frac{2\alpha}{n}.
$$
Note that $\alpha / n \to 0$ as $n \to \infty$ and hence we can take Taylor expansion of the left hand side with respect to $\lambda$ around zero
$$
\frac{\sigma^2 \lambda^2}{1 + \sigma^2 \lambda^2 / 2} = \frac{2\alpha}{n}.
$$

This argument gives us an approximate solution $\lambda = \frac{1}{\sigma} \sqrt{\frac{2\alpha}{n}}$. Similarly, we get the same result for the second term. Combining these results and applying Stirling's approximation, we get
\begin{equation*}
    \begin{split}
        \E S^{2 \alpha} &\le \Gamma(2\alpha+1)  \cdot 2 \frac{[1 + \frac{\sigma^2}{2} \frac{2}{\sigma^2} \frac{ \alpha }{n} + o(1)  ]^n}{2^{\alpha} \alpha^{\alpha}} n^{\alpha } \sigma^{2\alpha} \\
        &\le 2 \cdot \Gamma(2\alpha+1) \frac{e^{2\alpha  } (1 + o(1))}{2^{\alpha} \alpha^{\alpha} } n^{\alpha} \sigma^{2\alpha} \\
        &= 2(1 + o(1)) e^{1/24} \frac{e^\alpha}{2^{\alpha} \alpha^{\alpha}}\frac{(2\alpha)^{2\alpha}}{e^{2\alpha}}\sqrt{4\pi \alpha} \cdot n^\alpha \sigma^{2\alpha}  \\
        &< 4 \cdot 2^{\alpha} \Gamma(\alpha+1) n^{\alpha} \sigma^{2\alpha} . 
    \end{split}
\end{equation*}
\end{proof}

\section{Omitted proofs for finite groups}\label{asec:proof_finite}

We recall the notation for the proofs in this section. For  $\alpha = (\alpha_1, \dots, \alpha_k) \in \bb N^k$ denote by $T_{k, \alpha}$ the product of first $k$ terms as in \eqref{eq:lowdegree_ratio_with_n_ell} where $\ell$-th term is taken to the $\alpha_\ell$-th power:
\[
T_{k, \alpha} := \prod_{\ell = 1}^{k}  \lt | \frac{n - \sum\limits_{j=1}^{\ell-1} n_{j} }{L - \ell + 1} - n_{\ell} \rt|^{2\alpha_{\ell}},
\]
and constant $C_{k, \alpha}$ as
$$
C_{k,\alpha} := \prod_{\ell = 1}^{k} \lt ( \frac{L - \ell + 1}{
L - \ell } \rt) ^{\alpha_\ell}.
$$

The low-degree likelihood ratio is then expressed in terms of $T_{L-1, \mathbf d}$ as
\begin{equation}\label{eq:ldlr_T}
\| L^{\le D}_n \|^2 \le \sum_{d=0}^D \frac{\lambda^{2d}}{n^d d!}  \lt ( \frac{L}{2}\rt)^d \sum_{\substack{d_1, \dots, d_{L-1} \ge 0 \\ d_1 + \dots d_{L-1} = d}} \binom{d}{d_1 \dots d_{L-1}} C_{L-1,(d_1, \dots, d_{L-1})}  \E_{n_1, \dots, n_{L-1}}T_{L-1, (d_1, \dots, d_{L-1})}. 
\end{equation}

In the proofs throughout we will additionally use notation for rounding operations. We denote by $\ceil{\alpha}$ and $\floor \alpha $ the ceiling and flooring functions of $\alpha \in \mathbb{R}$ respectively. Furthermore, to avoid confusion with set notation, we deviate from the standard notation for the fractional part and instead denote it by $\fra(\alpha) := \alpha - \floor{\alpha}$.

The following lemma provides a recurrent expression for conditional expectation of $T_{k, \alpha}$. %

\begin{lemma}\label{lm:T_recurrent}
    Let $ 1 \le k < L$ be an integer. Assume that $\alpha \in \bb R_+^k$ is a vector with non-negative entries such that each $\alpha_i \in o(n)$ for $1 \le i \le k$. For any $\gamma \in \bb R_+$, almost surely it holds
    \begin{equation} \label{eq:T_recurrent}
        \begin{split}
            &\E_{n_k} \big [ T_{k, \alpha} \cdot \lt(n - \sum\limits_{j=1}^{k-1}n_j\rt)^{\gamma}  | n_1, \dots, n_{k-1}\big ] \\
            &\quad \le  2^{\alpha_{k}} \Gamma(\alpha_{k}+1) \sum_{\beta_{k}=0}^{\ceil{\alpha_k + \gamma}} \binom{\ceil{\alpha_k+\gamma}}{\beta_k} \lt (\frac{L - k}{(L - k + 1)^2}\rt)^{\alpha_k}  \lt(\frac{L-k+1}{L-k+2}\rt)^{\ceil{\alpha_k + \gamma} - \beta_k}  \\
            &\qquad \quad \lt(n - \sum\limits_{j=1}^{k-2}n_j\rt)^{\ceil{\alpha_k + \gamma} - \beta_k} T_{k-1, (\alpha_1, \dots, \alpha_{k-1} + \beta_{k-1}/2)}. 
        \end{split}
    \end{equation}
\end{lemma}
\begin{remark}
Observe that for $\gamma = 0$, \eqref{eq:T_recurrent} gives a recurrent expression for conditional expectation of $T_{k, \alpha}$, and after one step we get a product of $T_{k-1, \tilde \alpha}$ and a term that depends on $n_1, \dots, n_{k-2}$. Taking conditional expectation of this product corresponds to the left hand side of \eqref{eq:T_recurrent} for $k \leftarrow k - 1$ with $\gamma = \alpha_k - \beta_k$.
\end{remark}
\begin{proof}
Note that for every $1 \le j\le L-1$, $T_{j, \alpha}$ depends on $n_1, \dots, n_j$, hence we can write
\begin{equation*}
    \begin{split}
        &\E_{n_k} \bigl [ T_{k, \alpha} \cdot \lt(n - \sum\limits_{j=1}^{k-1}n_j\rt)^\gamma  \bigg | n_1, \dots, n_{k-1}\bigr ] \\
        &\quad = T_{k-1, (\alpha_1, \dots, \alpha_{k-1})} \cdot \lt(n - \sum\limits_{j=1}^{k-1}n_j\rt)^\gamma \\
&\qquad \E_{n_k} \lt [ \lt | \frac{n - \sum\limits_{j=1}^{k-1} n_{j} }{L - k + 1} - n_k \rt|^{2 \alpha_{k}} \bigg | n_1, \dots, n_{k-1}\rt ]
    \end{split}
\end{equation*}

Conditioned on $n_1, \dots, n_{k-1}$, random variable $n_{k}$ is distributed as sum of Bernoulli random variables $\xi_j, j = 1, \dots, n - \sum_{\ell=1}^{k-1} n_\ell$ with success probability $1 / (L - k + 1)$. Using this notation we can write 
\[
\E_{n_k} \lt [ \lt | \frac{n - \sum\limits_{j=1}^{k-1} n_{j} }{L - k + 1} - n_k \rt|^{2 \alpha_{k}} \bigg | n_1, \dots, n_{k-1}\rt ] = \E \lt ( \sum_{j=1}^{n - \sum_{\ell=1}^{k-1} n_\ell} [ \xi_j - \E \xi_j ]  \rt)^{2\alpha_k}.
\]
The variance of a single variable $\xi_j$ is
$(L - k) / (L - k + 1)^2$, and hence applying \Cref{prop:clt_moments} yields
\begin{equation}\label{eq:T_step_before_binomial}
    \begin{split}
        &\E_{n_k} \bigl [ T_{k, \alpha} | n_1, \dots, n_{k-1}\bigr ] \cdot \lt(n - \sum\limits_{j=1}^{k-1}n_j\rt)^\gamma \\
        &\quad \le T_{k-1, (\alpha_1, \dots, \alpha_{k-1})}  \\
&\qquad 2^{\alpha_{k}} \Gamma(\alpha_{j}+1) \lt (\frac{L - k}{(L - k + 1)^2}\rt)^{\alpha_k} \lt(n - \sum\limits_{j=1}^{k-1} n_{j}\rt)^{\alpha_k + \gamma} \\
&\quad\le T_{k-1, (\alpha_1, \dots, \alpha_{k-1})}  \\
&\qquad 2^{\alpha_{k}} \Gamma(\alpha_{j}+1) \lt (\frac{L - k}{(L - k + 1)^2}\rt)^{\alpha_k}  \lt(n - \sum\limits_{j=1}^{k-1} n_{j}\rt)^{\floor{\alpha_k + \gamma}} n^{\fra(\alpha_k + \gamma)},
    \end{split}
\end{equation}
where $\fra(\beta) = \beta - \floor{\beta}$ is a fractional part of $\beta \in \R$.

We simplify the expression and apply the binomial theorem
\begin{equation*}
    \begin{split}
        \E_{n_k} \bigl [ T_{k, \alpha} | n_1, \dots, n_{k-1}\bigr ] &\le T_{k-1, (\alpha_1, \dots, \alpha_{k-1})}
        2^{\alpha_{k}} \Gamma(\alpha_{j}+1) n^{\fra(\alpha_k + \gamma)} \lt (\frac{L - k}{(L - k + 1)^2}\rt)^{\alpha_k}  \\
        &\quad \sum_{\beta_k = 0}^{\floor{\alpha_k+\gamma}} \binom{\floor{\alpha_k + \gamma}}{\beta_k} \lt|\frac{n - \sum\limits_{j=1}^{k-2}n_j }{L-k+2} - n_{k-1}\rt|^{\beta_k} \\
        &\quad \lt(\frac{L-k+1}{L-k+2}\rt)^{\floor{\alpha_k + \gamma} - \beta_k} \lt(n - \sum\limits_{j=1}^{k-2}n_j\rt)^{\floor{\alpha_k + \gamma} - \beta_k} \\
        &= 
        2^{\alpha_{k}} \Gamma(\alpha_{j}+1) n^{\fra(\alpha_k + \gamma)} \lt (\frac{L - k}{(L - k + 1)^2}\rt)^{\alpha_k}  \\
        &\quad \sum_{\beta_k = 0}^{\floor{\alpha_k+\gamma}} \binom{\floor{\alpha_k + \gamma}}{\beta_k} T_{k-1, (\alpha_1, \dots, \alpha_{k-1} + \beta_k / 2)} \\
        &\quad \lt(\frac{L-k+1}{L-k+2}\rt)^{\floor{\alpha_k + \gamma} - \beta_k} \lt(n - \sum\limits_{j=1}^{k-2}n_j\rt)^{\floor{\alpha_k + \gamma} - \beta_k},
    \end{split}
\end{equation*}
which concludes the proof.
\end{proof}

\begin{lemma}\label{lm:expectation_T_step_k}
    For any $k \in [L-1]$, we have the following bound on expectation
    \begin{equation}\label{eq:low_degree_induction}
    \begin{split}
        &C_{L-1, \mathbf d} \E T_{L-1, \mathbf d}  \\
        &\quad \le C_{L-k-1, \mathbf d_{:L-k-1}} \sum_{\beta_{L-1}, \dots, \beta_{L-k}} \prod_{j=L-k}^{L-1} 2^{d_{j}+ \beta_{j+1}/2} \binom{d_{j} + \floor[\bigg]{\sum\limits_{t=j+1}^{L-1} d_{t} - \beta_{t}/2}}{\beta_{j}} \\
        &\qquad\quad \Gamma(d_{j} + \beta_{j+1}/2+1) n^{\sum_{t=L-k+1}^{L-1} \fra(\beta_j/2)} (L-j)^{-\beta_{j+1}/2 - \floor{\beta_{j+2}/2}}\\
        &\qquad\quad (k+1)^{-\floor{\beta_{L-k+1}/2} - \beta_{L-k}}
         (k+2)^{-\gamma_k}   \\
        &\qquad \quad \E \lt(n - \sum\limits_{t=1}^{L-k-2} n_t \rt)^{\gamma_k}   T_{L-k-1, (\mathbf d_{:L-k-2}, d_{L-k-1} + \beta_{L-k} / 2)}
    \end{split}
    \end{equation}
where $\gamma_k := d_{L-k} + \sum\limits_{t=L-k+1}^{L-1} \lt(d_t - \ceil{\beta_t/2}\rt) - \beta_{L-k}, \beta_L, \beta_{L+1} := 0$, and the first sum ranges over $\beta_{j} = 0, \dots, d_{j} + \floor[\bigg]{\sum\limits_{t=j+1}^{L-1} d_{t} - \beta_{t}/2}$ for $j = L-k, \dots, L-1$. 
\end{lemma}
\begin{proof}
We prove the statement by induction.
 By~\Cref{lm:T_recurrent}, after the first step, we obtain
\begin{equation*}
    \begin{split}
        &C_{L-1, \mathbf d} \E T_{L-1, \mathbf d} = C_{L-1, \mathbf d} \E_{\mathbf{n}_{1:L-2}} \lt [ \E_{n_{L-1}} \lt [ T_{L-1, \mathbf d} | n_1, \dots, n_{L-2} \rt ] \rt ] \\
        &\quad \le C_{L-2, \mathbf d_{:L-2}} \cdot 2^{d_{L-1}} d_{L-1}! \\
        &\qquad \cdot \sum_{\beta_{L-1}=0}^{d_{L-1}} \binom{d_{L-1}}{\beta_{L-1}} \E_{\mathbf n_{1:L-2}} T_{L-2, (\mathbf d_{:L-3}, d_{L-2} + \beta_{L-1} / 2)} \frac{2^{-\beta_{L-1}}}{3^{d_{L-1}-\beta_{L-1}}} \lt ( n - \sum_{j=1}^{L-3} n_j \rt)^{d_{L-1}-\beta_{L-1}}.
    \end{split}
\end{equation*}

Fix $k \ge 2$, and assume that \eqref{eq:low_degree_induction} holds for $k$ by induction hypothesis. To perform an induction step for $k+1$ we use tower property and \Cref{lm:T_recurrent} as follows.   %

\begin{equation}\label{eq:induction_step}
\begin{split}
&\E \lt(n - \sum\limits_{t=1}^{L-k-2} n_t \rt)^{\gamma_k}  T_{L-k-1, (\mathbf d_{:L-k-2}, d_{L-k-1} + \beta_{L-k} / 2)}  \\
&\quad = \E  \lt(n - \sum\limits_{t=1}^{L-k-2} n_t \rt)^{\gamma_k} \E_{n_{L-k-1}} \lt [T_{L-k-1, (\mathbf d_{:L-k-2}, d_{L-k-1} + \beta_{L-k} / 2)}  \big | n_{1:L-k-2} \rt ] \\
&\quad \le 2^{d_{L-k-1} + \beta_{L-k} / 2} \Gamma(d_{L-k-1} + \beta_{L-k} / 2+1) n^{\fra(\beta_{L-k}/2)} \lt ( \frac{k+1}{(k+2)^2} \rt)^{d_{L-k-1} + \beta_{L-k}/2} \\
&\qquad \quad  \sum_{\beta_{L-k-1}} \binom{ d_{L-k-1} + \floor[\bigg]{\sum\limits_{t=L-k}^{L-1}d_t - \beta_t/2} }{\beta_{L-k-1}} \lt ( \frac{k+2}{k+3}\rt)^{d_{L-k-1} + \floor{\beta_{L-k}/2 + \gamma_k} - \beta_{L-k-1}}  \\
&\qquad \quad \E T_{k-1, (d:L-k-3, d_{L-k-2} + \beta_{L-k-1}/2)} \lt(n - \sum\limits_{t=1}^{L-k-3} n_t \rt)^{ d_{L-k-1} + \floor{\beta_{L-k}/2 + \gamma_k} - \beta_{L-k-1}} ,
\end{split}
\end{equation}
where $\beta_{L-k-1}$ ranges from $0$ to $\sum\limits_{t=L-k}^{L-1}( d_t - \floor{\beta_t/2}) + d_{L-k-1} $. 

After substituting the above expression into the right-hand side of \eqref{eq:low_degree_induction} for $k$, it can be easily seen that the expression is multiplied by the corresponding binomial coefficient, gamma function, and a power of 2. Consequently, it remains to compute the exponents of the constants in order to complete the proof.
\begin{align*}
    k+1: \quad &- 
    \underbrace{\floor{\beta_{L-k+1}/2} - \beta_{L-k}}_{\text{r.h.s. of \eqref{eq:low_degree_induction}}} + \underbrace{d_{L-k-1} + \beta_{L-k}/2}_{\text{step $k$ in \eqref{eq:induction_step}}}  -\underbrace{d_{L-k-1}}_{\text{from } C_{L-k-1, \mathbf d_{:L-k-1}}}  \\
    &= - \beta_{L-k}/2 - \floor{\beta_{L-k+1}/2}; \\
    k+2: \quad &\underbrace{ -\gamma_k }_{\text{step $k$ in \eqref{eq:induction_step}}} + \underbrace{d_{L-k-1} + \gamma_k + \ceil{\beta_{L-k}/2} - \beta_{L-k-1}}_{\text{step $k+1$ in \eqref{eq:induction_step}: binomial coefficient}}- \underbrace{(2 d_{L-k-1} + \beta_{L-k})}_{\text{step $k+1$ in \eqref{eq:induction_step}: variance}} + \underbrace{d_{L-k-1}}_{\text{from } C_{L-k-1, \mathbf d_{:L-k-1}}} \\
    &= -\floor{\beta_{L-k}/2} - \beta_{L-k-1}  ;\\
    k+3: \quad & - d_{L-k-1} \underbrace{- d_{L-k} - \sum_{t=L-k+1}^{L-1} (d_{t} - \floor{\beta_t/2}) + \beta_{L-k}}_{\gamma_k} - \ceil{\beta_{L-k}/2} + \beta_{L-k-1} \\
    &=- d_{L-k-1} -  \sum_{t=L-k}^{L-1} (d_t - \floor{\beta_t/2}) + \beta_{L-k-1} = -\gamma_{k+1}.
\end{align*}

\end{proof}

\begin{corollary}\label{cor:expectation_T_full} For $L > 2$, the full expectation is bounded as follows. 
\begin{equation*}
    \begin{split}
        &C_{L-1, \mathbf d} \E T_{L-1, \mathbf d} \le \frac{2^d}{L^d} \sum_{\beta_2, \dots, \beta_{L-1}} \Gamma(d_1 + \beta_2 / 2 +1 ) (L-1)^{-\beta_2/2 - \floor{\beta_3}/2}\\
        &\qquad\quad  \prod_{j=2}^{L-1} 2^{\beta_{j+1}/2} \binom{d_{j} + \sum\limits_{t=j+1}^{L-1} d_{t} - \floor{\beta_{t}/2}}{\beta_{j}} 
        \Gamma(d_{j} + \beta_{j+1}/2+1) (L-j)^{-\beta_{j+1}/2 - \floor{\beta_{j+2}/2}} n^{-\floor{\beta_j / 2}} L^{-\floor{\beta_{j}/2}},
    \end{split}
\end{equation*}
where $\beta_{L}, \beta_{L+1} \coloneqq 0 $ and $\beta_{j}$ ranges from $0$ to $d_{j} + \ceil[\bigg]{\sum\limits_{t=j+1}^{L-1} d_{t} - \beta_{t}/2}$ for $j = L-k, \dots, L-1$.
\end{corollary}

\begin{proof}
    We apply \Cref{lm:expectation_T_step_k} with $k=L-2$. After this step the only non-deterministic term is $T_{1, (d_1 + \beta_2/2)}$. To bound the last expectation, we use inequality \eqref{eq:T_step_before_binomial} with $k=1$ and $\gamma = \gamma_{L-2}$ and compute the exponents similarly as in the proof of \Cref{lm:expectation_T_step_k}. 
    \begin{align*}
        L: \quad & \underbrace{-2d_1 - \beta_2 / 2}_{\text{Variance}} - \underbrace{\lt(\sum_{t = 2}^{L-1} d_t - \sum_{t=3}^{L-1} \floor{\beta_t/2} - \beta_2\rt)}_{\gamma_{L-2}} + \underbrace{d_1}_{C_{1, d_1}} \le -d + \sum_{t=2}^{L-1} \floor{\beta_t/2};\\
        L-1:\quad& -\floor{\beta_3/2} - \beta_2 + d_1 - \beta_2/2 - d_1 = -\floor{\beta_3/2} - \beta_2/2;\\
        n: \quad & \underbrace{\lt(\sum_{t = 2}^{L-1} d_t - \sum_{t=3}^{L-1} \ceil{\beta_t/2} - \beta_2\rt)}_{\gamma_{L-2}} + d_1 + \beta_2 / 2 + \sum_{t=3}^{L-1} \fra(\beta_t/2) = d - \sum_{t=2}^{L-1} \beta_t/2.
    \end{align*}
    In the calculation of power of $n$, we used that $\beta_t \in \bb Z_+$, and therefore for odd $\beta_t$ the fractional part $\fra(\beta_t/2) = 1/2 = \ceil{\beta_t/2} - \beta_t/2$.  
\end{proof}

\begin{proof}[Proof of \Cref{thm:main_all_priors}]
Using the expression of $\| L^{\le D}_n \|^2$ from \eqref{eq:ldlr_T} and \Cref{cor:expectation_T_full} we write
\begin{align}
        \| L^{\le D}_n \|^2 &=\sum_{d=0}^D \frac{\lambda^{2d}}{n^d d!}  \lt ( \frac{L}{2}\rt)^d  \sum_{\substack{d_1, \dots, d_{L-1} \ge 0 \\ d_1 + \dots d_{L-1} = d}} \binom{d}{d_1 \dots d_{L-1}} C_{L-1,(d_1, \dots, d_{L-1})}  \E_{n_1, \dots, n_{L-1}}T_{L-1, (d_1, \dots, d_{L-1})} \nonumber \\
        &\le \sum_{d=0}^D \lambda^{2d}  \sum_{\substack{d_1, \dots, d_{L-1} \ge 0 \\ d_1 + \dots d_{L-1} = d}} \sum_{\beta_2, \dots, \beta_{L-1}} \frac{\Gamma(d_1 + \beta_2 / 2 +1 )}{d_1!} (L-1)^{-\beta_2/2 - \floor{\beta_3}/2} \nonumber \\
        &\qquad \prod_{j=2}^{L-1} 2^{\beta_{j+1}/2} \binom{d_{j} + \sum\limits_{t=j+1}^{L-1} d_{t} - \floor{\beta_{t}/2}}{\beta_{j}} \frac{\Gamma(d_j + \beta_{j+1}/2)}{d_j!}\nonumber \\
        &\qquad \frac{1}{(n L)^{\floor{\beta_{j}/2}}} \frac{1}{(L-j)^{-\beta_{j+1} - \floor{\beta_{j+2}/2}}} \nonumber \\
        &\le \sum_{d=0}^D \lambda^{2d}  \sum_{\substack{d_1, \dots, d_{L-1} \ge 0 \\ d_1 + \dots d_{L-1} = d}} \sum_{\beta_2, \dots, \beta_{L-1}} \prod_{j=2}^{L-1}  \frac{ 2^{\beta_{j}/2}d^{\beta_j} (2d)^{\beta_j/2} }{ n^{\beta_j/2}}. \label{eq:triple_sum} %
\end{align}
Here we used the following bounds on Gamma function 
$$
\frac{\Gamma(d_j + \beta_{j+1}/2 +1)}{d_j!} \le (d_j + \beta_{j+1}/2)^{\beta_j/2} \le (2d)^{\beta_j/2}
$$
and on binomial coefficient 
$$
\binom{d_{j} + \sum\limits_{t=j+1}^{L-1} d_{t} - \floor{\beta_{t}/2}}{\beta_{j}} \le \lt(d_{j} + \sum\limits_{t=j+1}^{L-1} d_{t} - \floor{\beta_{t}/2}\rt)^{\beta_j} \le d^{\beta_j}.
$$
Moreover, we used the fact that $\beta_{L} = 0$ by definition and therefore $2^{\beta_L} = 1$. This allowed us to replace $2^{\beta_{j+1}}$ with $2^{\beta_j}$ after additionally multiplying by the starting term $2^{\beta_{3}}$.

We proceed to bounding the product in \eqref{eq:triple_sum}. Note that when at least one $\beta_j \ne 0$ we have that 
$$
\prod_{j=2}^{L-1}  \frac{ 2^{\beta_{j}/2}d^{\beta_j} (2d)^{\beta_j/2} }{ n^{\beta_j/2}} = \prod_{j=2}^{L-1} \lt( \frac{ 4 d^3 }{ n} \rt)^{\beta_j/2} \le 1,
$$
since $d = o(n^{1/3}) $. Hence we can bound the two last sums in \eqref{eq:triple_sum} with its maximum term as 
$$
\| L^{\le D}_n \|^2 \le \sum_{d=0}^D \lambda^{2d}  \sum_{\substack{d_1, \dots, d_{L-1} \ge 0 \\ d_1 + \dots d_{L-1} = d}} \sum_{\beta_2, \dots, \beta_{L-1}} d^L = \sum_{d=0}^D \lambda^{2d} d^{2L} \le \sum_{d=0}^\infty \lambda^{2d} d^{2L},
$$
completing the proof.

\end{proof}

\end{document}